\documentclass[12pt]{article}
\usepackage{amsmath,amssymb,amsthm,tikz,dsfont}
\usepackage{bm}
\usepackage{graphics}
\usepackage[margin=1in]{geometry}
\newcommand{\bfnu}{\boldsymbol{\nu}}
\newcommand{\bfx}{{\bm x}}
\newcommand{\bfy}{{\bm y}}
\newcommand{\bfz}{{\bm z}}
\newcommand{\bfd}{{\bm d}}
\newcommand{\hbfx}{\hat{\bfx}}

\newcommand{\demi}{\frac{1}{2}}
\newcommand{\cF}{{\cal{}F}}
\newcommand{\jmp}[1]{\left[\!\left[#1\right]\!\right]}   
\newtheorem{theorem}{Theorem}

\newtheorem{remark}{Remark}
\newtheorem{lemma}{Lemma}
\newtheorem{definition}{Definition}

\newtheorem{assumption}{Assumption}
\newtheorem{prop}{Proposition}
\usepackage{hyperref}
\usepackage[nameinlink]{cleveref}

\begin{document}
\title{Target signatures for thin surfaces}
\author{Fioralba Cakoni\footnote{Department of Mathematics, Rutgers University, New Brunswick,  New Jersey,  USA 
 (fc292@math.rutgers.edu)} \and Peter Monk\footnote{Department of Mathematical Sciences, University of Delaware, Newark,  Delaware,  USA 
 (monk@udel.edu)} \and Yangwen Zhang\footnote{Department of Mathematical Sciences, University of Delaware, Newark,  Delaware,  USA 
 (ywzhangf@udel.edu)}}
\maketitle

 \begin{abstract}
	We investigate an inverse scattering problem for a thin inhomogeneous scatterer in ${\mathbb R}^m$, $m=2,3$, which we model as a $m-1$ dimensional open surface. The scatterer is referred to as a screen. The goal is to design  target signatures that are  computable from scattering data in order to detect changes in the material properties of the screen. This target signature is characterized by a mixed Steklov eigenvalue problem for  a domain whose boundary contains the screen. We show that the corresponding eigenvalues  can be determined  from appropriately modified  scattering data by using the generalized linear sampling  method. A weaker justification is provided for the classical linear sampling method. Numerical experiments are presented to support our theoretical results.
 \end{abstract}

\noindent{\bf Key words:}  inverse scattering, inhomogeneous media,  scattering by screens, non-destructive testing, the Steklov eigenvalue problem.\\
\noindent{\bf AMS subject classifications:} 35R30, 35J25, 35P25, 35P05

\section{Introduction}

Target signatures are discrete quantities that can be computed from scattering data and used to classify targets or serve as indicators of changes in targets.  A classical choice is the Singularity Expansion Method introduced to classify targets in radar scattering proposed by C. Baum~\cite{baum_91,Melrose95}, although this method eventually turned out to be difficult to determine in the radar application since highly accurate time domain data is needed.  More recently, Cakoni, Colton and Haddar suggested that transmission eigenvalues might provide target signatures for detecting changes in an object using multi-frequency  time harmonic data (see \cite{CCH16} for an introduction to transmission eigenvalues and the target signatures).   By sweeping through the frequency data, transmission eigenvalues can be determined from the far field operator for the scattered data provided the medium is non-conducting.  For a conducting medium the eigenvalues may be complex, and hence cannot be detected using real probing frequencies.

In an attempt to circumvent this problem, Cakoni, Colton, Meng and Monk~\cite{CCMM2016} introduced a modified far field operator by subtracting the far field operator for scattering by an artificial domain containing the scatterer and having an impedance boundary condition with variable impedance.  In this case the target signatures correspond to Steklov eigenvalues for the artificial domain and scatterer, and these can be determined using data at a single frequency.  Furthermore, the method can be applied to conducting objects since the artificial impedance problem can have a complex parameter.
The use of Steklov eigenvalues was later extended to Maxwell's equations \cite{CLM2017}, and the theoretical analysis was greatly improved by
giving a general framework for determining target signatures from far field data using the Generalized Linear Sampling method to determine the Steklov eigenvalues~\cite{ACH2017}.  This latter paper also introduced a new class of target signatures based on an artificial  modified interior transmission problem. All the previously mentioned papers had the disadvantage that it is not known if Steklov eigenvalues exist for conducting media (but if they do, they can be detected!).  This situation was remedied by Cogar~\cite{C2020} who further modified the Steklov problem by introducing a smoothing operator
into the impedance boundary condition.  This allows the existence of complex Steklov type eigenvalues for the new boundary condition to be proved.

All the previously mentioned work has involved scattering from bodies with a nontrivial interior (i.e. containing a disk in $\mathbb{R}^2$ or a sphere in $\mathbb{R}^3$).  However if the scatterer is thin compared to the wavelength of the probing radiation, it is often desirable to treat it as a screen with zero thickness but having special transmission conditions across the surface that model, approximately, the thin body.  In this paper we will  use a special case of the transmission condition developed in \cite{Irene_16} (see also \cite{DHJ_12}) for delaminating media.

An obvious problem with developing a Steklov type target signature for a screen is that there is no interior.  Instead, we propose to use an artificial domain
and a mixed Steklov eigenvalue problem.  In particular, suppose the screen is denoted by $\Gamma$; we assume that we can extend $\Gamma$ to obtain a closed piecewise smooth surface that is the boundary of a bounded connected region $D\subset \mathbb{R}^m$.  Thus
$\Gamma\subset \partial D$.  We then modify the far field pattern using a mixed scattering problem in the exterior of $D$ having a variable impedance 
denoted $\lambda$.  Modifying the measured far field pattern using the far field pattern from this artificial scattering problem allows us to demonstrate a
connection between the far field operator and an interior mixed Steklov eigenvalue problem posed on $D$.  In particular, the mixed Steklov eigenvalues can be determined from the measured far field pattern due to scattering by the screen.  We further connect properties of the surface impedance of the
screen to the target signatures.  Finally, we provide some numerical results in 2D that illustrate the theory. 

The choice of the domain $D$ is somewhat arbitrary beyond the constraint that $\Gamma\subset \partial D$. Our limited numerical results suggest that the choice can effect the number and quality of the eigenvalues that can be determined.  For the simple curves in 2D that we have considered it is sufficient to use the convex hull of the points on $\Gamma$.  Obviously this is not appropriate for more general screens.

This is the first paper on Steklov type target signatures for screens.  A related work on modified transmission target signatures for cracks can be found in 
\cite{C2018}. We remark that the method presented here aims to detect changes in the material properties of the screen  represented by coefficients in the jump conditions of the total field, and assumes that the geometry of the open surface is known. Most of available literature concerns the reconstruction of the shape of an open surface, and  we refer the reader to  \cite{crack3}, \cite{crackH}, \cite{crack4}, \cite{crack2}  for some related non-iterative inversion methods.

The paper proceeds as follows.  In the next section define the direct (or forward) problem and our proposal for target signatures. Then in \Cref{sec:deff} we show that the mixed Steklov eigenvalues defined in \Cref{sec:DAIP} can be determined using the Linear Sampling Method (LSM) or the generalized LSM.  As usual, the justification of the use of the convenient LSM is incomplete (for more discussion of this point, see \cite{CCH16} and comments later in this paper), while for the generalized LSM a complete theory can be given.  Nevertheless we use the LSM for our numerical results.  Results 
connecting the mixed Steklov eigenvalues to the impedance of $\Gamma$ are provided in \Cref{sec:resp} where we note, for example, that a
constant impedance can be determined from measurements of a single eigenvalue.  Next, in \Cref{sec:num}, we show preliminary numerical
results of determining eigenvalues from far field data, and investigate the sensitivity of the eigenvalues to perturbations of the impedance.  Conclusions are  presented in \Cref{sec:concl}.  The paper ends with an appendix (\ref{sec:exact}) on the exact calculation of mixed Steklov eigenvalues relevant to our
study.

\section{The Direct, Auxiliary and Inverse Problem}\label{sec:DAIP}

In this section we define the forward problem and target signatures.
We assume that $\Gamma$ is an open piecewise smooth $m-1$ dimensional,  oriented and non-self-intersecting surface in ${\mathbb R}^m$ for $m=2,3$. More precisely, we consider $\Gamma$ to be a subset of a closed piecewise smooth surface $\partial D$ circumscribing a connected region $D\subset {\mathbb R}^m$. Let  $\bfnu$ denote the unit outward normal to $\partial D$ and {let} the normal derivative $\partial_{\bfnu} u = \nabla u \cdot \bfnu$.  For a piecewise smooth function $w$ we define the jump on  $\Gamma$ by
{\[
	\jmp{w}=w|_{\mathbb{R}^m\setminus \overline{D}}-w|_{D}.
	\]}

The surface material parameter $\sigma:={k(\alpha_1+i\alpha_2)}$  is a complex valued  $L^{\infty}(\Gamma)$ function with non-negative imaginary part  ${\alpha_2\geq 0}$, where $k$ denotes the wave number of the background medium which is proportional to the interrogating frequency. Given an incident field $u^i$ and the surface parameter $\sigma$, the forward problem for $\Gamma$ is to find {$u\in H^1_{\rm{}loc}(\mathbb{R}^m)$} such that
\begin{subequations}
	\begin{eqnarray}
		\Delta u+k^2 u==0  \mbox{ in } \mathbb{R}^m\setminus \Gamma, &&  u=u^s+u^i  \mbox{ in }\mathbb{R}^m, \\
		\jmp{\partial_{\bm \nu} u}{+\sigma u}=0 \mbox{ on }\Gamma,  && \jmp{u}=0  \mbox{ on }\Gamma,\\
		\lim_{r\to\infty}r^{\frac{m-1}{2}}\left({\partial_r u^s}-iku^s\right)=0 && \mbox{ uniformly in }\hbfx=\bfx/r, \;r=\Vert\bfx\Vert.
	\end{eqnarray}\label{Forward_problem}
\end{subequations}

In practice, we usually take the incident field $u^i$ to be plane wave given by $u^i(\bfx)=\exp(i k \bfd\cdot\bfx)$, where we denote the incident direction by the unit vector $\bfd$.  
The far field pattern $u_{\infty}(\hbfx;\bfd)$ for  $\Vert \hbfx\Vert=1$ is defined from the following asymptotic behavior of the scattered field {\cite{CK2019}}
\begin{equation}
	u^s(\bfx)=\frac{\exp(ikr)}{r^{\frac{m-1}{2}}}\left\{u_{\infty}(\hbfx;\bfd)+O\left(\frac{1}{r^{\frac{m-1}{2}}}\right)\right\}\mbox{ as } r\to\infty.\label{def_us}
\end{equation}

\begin{remark}
	The {\it inverse problem} we are concerned with is, provided that  the shape $\Gamma$ of the surface is known, determine indicators of changes in the surface material parameter  $\sigma$ from a knowledge of the far field pattern $u_{\infty}(\hbfx;\bfd)$  for observation directions $\hbfx$ and incident directions $\bfd$ on the unit sphere or circle ${\mathbb{S}}:=\left\{\bfx\in {\mathbb R}^m: \, \Vert \bfx\Vert=1\right\}$ at a fixed wave number $k$. 		
\end{remark}


Before moving on to the inverse problem, we recall the radiating fundamental solution $\Phi(\bm x, \bm z)$ to the Helmholtz equation
\begin{equation}\label{fund}
	\Phi(\bm x, \bm z):=\left\{\begin{array}{rrcll}\displaystyle{\frac{e^{ik|\bfx-\bfz|}}{4\pi |\bfx-\bfz|}} \quad \; & \qquad  \mbox{in }\, {\mathbb R}^3, \\
		\displaystyle{\frac{i}{4}}H_0^{(1)}(k|\bfx-\bfz|) & \qquad \mbox{in }\, {\mathbb R}^2, 
	\end{array}\right.
\end{equation}
where  $H^{(1)}_0$ denoting the Hankel function of the first kind of order zero. This will be used in the upcoming analysis.

Our approach is based on the development of a target signature for $\sigma$  that uses the eigenvalues of an appropriate eigenvalue problem. We remark that this target signature can detect changes on $\sigma$ without knowing the base healthy value  of $\sigma$  nor reconstructing it. To introduce the eigenvalue problem we need the following auxiliary scattering problem.

{Given $\lambda\in\mathbb{C}$ with $\Im (\lambda)\ge 0$}, the auxiliary scattering impedance problem that we consider is to find 
$w^{(\lambda)}\in H^1_{\rm{}loc}(\mathbb{R}^m\setminus \overline{D})$ such that 
\begin{subequations}
	\begin{eqnarray}
		\Delta w^{(\lambda)}+k^2 w^{(\lambda)}=0  \mbox{ in } \mathbb{R}^m\setminus \overline{D},&\quad &
		w^{(\lambda)}=w^{(\lambda),s}+u^i  \mbox{ in }{\mathbb{R}^m\setminus \overline{D}}, \\
		\partial_{\bfnu} w^{(\lambda)} +\lambda w^{(\lambda)}=0  \mbox{ on }\Gamma,&\quad & 
		\partial_{\bfnu} w^{(\lambda)} =0 \mbox{ on }\partial D\setminus \Gamma,\\
		\lim_{r\to\infty}r^{\frac{m-1}{2}}\left(\partial_r w^{(\lambda),s} -ikw^{(\lambda),s}\right)=0  &&\mbox{ uniformly in  }\hbfx=\bfx/r.
	\end{eqnarray}	\label{auxiliary_problem}	
\end{subequations}

As before,  the far field pattern $w^{(\lambda)}_\infty$ corresponding to the scattered field $w^{(\lambda),s}$ is defined by the asymptotic condition
\[
w^{(\lambda),s}(\bfx)=\frac{\exp(ikr)}{r^{\frac{m-1}{2}}}\left\{w^{(\lambda)}_{\infty}(\hbfx;\bfd)+O\left(\frac{1}{r^{\frac{m-1}{2}}}\right)\right\}.
\]
The far field pattern of the above auxiliary  problem can be computed and involves only the known geometry of the surface $\Gamma$, and surface impedance parameter $\lambda$.

From the known measured far field patten $u_{\infty}(\hbfx;\bfd)$, and the artificial far field pattern $w^{(\lambda)}_{\infty}(\hbfx;\bfd)$, we define the standard far field operators for  $g\in L^2(\mathbb S)$ by 
\[
(F g)(\hbfx)=\int_{\mathbb{S}}u_{\infty}(\hbfx;\bfd)g(\bfd)\,ds_{d},\quad 
(F^{(\lambda)} g)(\hbfx)=\int_{\mathbb{S}}w^{(\lambda)}_{\infty}(\hbfx;\bfd)g(\bfd)\,ds_{d}.
\]
We also define the {\it modified far field operator} $\cF:L^2(\mathbb{S})\to L^2(\mathbb{S})$ by 
\[
(\cF g)(\hbfx)=\int_{\mathbb{S}}(u^{s}_{\infty}(\hbfx;\bfd)-w^{(\lambda),s}_{\infty}(\hbfx;\bfd))g(\bfd)\,ds_{d}.
\]
Obviously, the modified far field operator is 
\begin{equation}\label{dif}
	\cF g =Fg-F^{(\lambda)}g.
\end{equation}

For theoretical purposes, we sometimes use the Herglotz wave function
\begin{equation}\label{herg}
	u^i_g(\bfx)=\int_{\mathbb{S}}\exp(ik\bfx\cdot\bfd)g(\bfd)\,ds_d
\end{equation}
for $g\in L^2(\mathbb{S})$ as the incident field $u^i$ in the forward {problem} (\ref{Forward_problem}) and auxiliary problem (\ref{auxiliary_problem}). Let  $u_g$ and $w_g^{(\lambda)}$ be the solution of  the two problems  by using this function. Furthermore, we write the corresponding  far field patterns as
$u_{g,\infty}$ and $w_{g,\infty}^{(\lambda)}$.
\begin{lemma}\label{pro_F}
	Let $u_{g,\infty}$ and $w_{g,\infty}^{(\lambda)}$ be defined as above, then $\cF g=u_{g,\infty}-w_{g,\infty}^{(\lambda)}$.
\end{lemma}
\noindent
{Throughout the remainder of the paper we make the following assumption on the domain $D$ in relation to the interrogating wave number $k$. 
	\begin{assumption}\label{assum}
		$k^2$ is not a mixed Dirichlet and Neumann eigenvalue  of $-\Delta$, i.e. of the problem
		\begin{equation}\label{mix}
			\Delta u+k^2u=0 \quad \textup{in} \; D, \qquad u=0\quad \textup{on}\; \Gamma, \qquad \partial_{\bfnu} u =0\quad \textup{on} \;  \partial D\setminus \Gamma,
		\end{equation}
	\end{assumption}
	\noindent
	We note that Assumption \ref{assum} is not a restriction since based on the interrogation frequency known to us, we can easily choose  $\partial D$ containing $\Gamma$ that satisfies  this assumption.}

\medskip

Next, we consider the following mixed Steklov eigenvalue problem:
\begin{equation}
	\Delta h+k^2h=0 \mbox{ in }D, \quad {\partial_{\bfnu} h} - \sigma h=-\lambda h \mbox{ on } \Gamma, \quad
	{\partial}_{\bfnu} h=0 \mbox{ on }\partial D\setminus \Gamma.\label{Steklov_eig}
\end{equation}
This is an extension to the Helmholtz equation of the standard sloshing problem.
To relate the  eigenvalue problem (\ref{Steklov_eig}) to the modified far field operator, we consider the injectivity of $\cF$.  
\begin{lemma}\label{F_injective}
	Assume that $\lambda\in {\mathbb C}$ is not an eigenvalue of (\ref{Steklov_eig}). Then the operator  $\cF:L^2(\mathbb{S})\to L^2(\mathbb{S})$  is injective
\end{lemma}
\begin{proof}
	Suppose $\cF g=0$ for some $g\in L^2(\mathbb{S})$, by \Cref{pro_F} we have  $u_{g,\infty}=w_{g,\infty}^{(\lambda)}$ on $\mathbb{S}$, and Rellich's Lemma~\cite[Lemma 2.12]{CK2019} implies $u_g^s=w_g^{(\lambda),s}$ in $\mathbb{R}^m\setminus D$. Adding  the Herglotz incident field $u_g^i$ to both scattered fields $u_g^s$ and $w_g^{(\lambda),s}$ we have $u_g = w^{(\lambda)}_g$ in $\mathbb{R}^m\setminus{\overline{D}}$. Denoting $u_g|_{\mathbb{R}^m\setminus{\overline{D}}}$ by $u_g^+$ and $u_g|_{D}$ by $u_g^-$, then {$u^+_g=w^{(\lambda)}_g$} in $\mathbb{R}^m\setminus{\overline{D}}$. Using the boundary conditions
	for $w_g^{(\lambda)}$ we have
	\begin{eqnarray*}
		{\partial_{\bfnu} u_g^+}+\lambda u_g^+=0 \mbox{ on }\Gamma,\quad 
		{\partial_{\bfnu} u_g^+}=0  \mbox{ on }\partial D\setminus \Gamma.
	\end{eqnarray*}	
	But on $\Gamma$ we have 
	\[
	u_g:=u_g^+=u_g^- \qquad \mbox{ and } \qquad  \partial_{\bfnu} u_g^+ -\partial_{\bfnu} u_g^- +\sigma u_g=0.
	\]
	Therefore, we have 
	\[
	\partial_{\bfnu} u_g^- =(\sigma-\lambda) u^-_g \mbox{ on }\Gamma,  \qquad 
	\partial_{\bfnu} u_g^- =\partial_{\bfnu} u_g^+ = 0 \mbox{ on }\partial D\setminus \Gamma.	
	\]
	This shows that $u_g^-$ is a solution of (\ref{Steklov_eig}).
	
	If  $\lambda\in {\mathbb C}$ is not an eigenvalue of the mixed Steklov eigenvalue problem (\ref{Steklov_eig}), then $u_g^-\equiv 0$ in $D$. In this case, the analyticity of the solution of the Helmholtz equation implies that we have that $u_g=0$ in ${\mathbb R}^m\setminus \Gamma$. In addition we obtain that the  jumps of $u_g$ and $\partial_{\bfnu} u_g$ are both zero, which means that $u_g$ solves the Helmholtz equation in ${\mathbb R}^m$ and $u^s_g=-u^i_g$, which is possible only if $g=0$ since $u^s_g$ is a {radiating solution whereas $u^i_g$ is an entire solution.}
	
	Therefore, if $\lambda$ is not an eigenvalue of (\ref{Steklov_eig}), the modified far field operator $\cF$ is injective. 
\end{proof}

\begin{remark}
	From \Cref{F_injective} we know  that   if $\cF g=0$  {has} a non-trivial solution, then $\lambda$  is an eigenvalue of problem (\ref{Steklov_eig}).  Note that the converse is not necessarily true, i.e. if $\lambda$ is an eigenvalue of (\ref{Steklov_eig}), this doesn't mean that $\cF$ is not injective, which will become clear in the following section. Nevertheless the above connection between the modified far field operator $\cF$ and the eigenvalue problem  (\ref{Steklov_eig}) can be exploited to  detect these eigenvalues  from the scattering data.  		
\end{remark}

\begin{lemma}\label{dense}
	Assume that $\lambda\in {\mathbb C}$ is not an eigenvalue of problem (\ref{Steklov_eig}). Then the range of $\cF:L^2(\mathbb{S})\to L^2(\mathbb{S})$ is dense.
\end{lemma}
\begin{proof}
	Let $\cF^*:L^2(\mathbb{S})\to L^2(\mathbb{S})$ be the adjoint of $\cF$, and define
	$R:L^2(\mathbb{S})\to L^2(\mathbb{S})$ by $Rg(\hbfx):=g(-\hbfx)$.  Since the far field patterns of of the solution of both  (\ref{Forward_problem}) and (\ref{auxiliary_problem}) satisfy a reciprocity
	relation (see  \cite[Theorem 8.8]{CK2019} for (\ref{Forward_problem}), a similar proof holds for (\ref{auxiliary_problem}), using (\ref{dif}) together 
	with the argument in \cite[Theorem 3.30]{CK2019} we have that
	$$\cF^* g=F^* g-F^{(\lambda)*}g={\overline{R(F-F^{(\lambda)})R\overline{g}}=\overline{R\cF R\overline{g}}}.$$
	Thus $\cF^*$ is injective if $\lambda\in {\mathbb C}$ is not an eigenvalue of (\ref{Steklov_eig},) which implies $\cF$ has dense range.
\end{proof}

Now we are in a position to define precisely the target signatures considered in this paper:
\begin{definition}[Target signatures for the screen $\Gamma$] Given a screen $\Gamma$ and a domain $D$ with $\Gamma\subset\partial D$ the target signature for the scatterer is the set of mixed Steklov eigenvalues defined by (\ref{Steklov_eig}).
\end{definition}

\section{Determination of the eigenvalues from far field data}\label{sec:deff}

We now show that mixed Steklov eigenvalues can be determined from far field data.  This involves a non-standard analysis of the exterior scattering problem.  We prove results for the convenient  Linear Sampling Method (LSM), and more precise results for the generalized LSM (GLSM).

The forward problem (\ref{Forward_problem}) is a particular case of the following transmission problem: given $\varphi \in H^{1/2}(\partial D)$, $\psi \in H^{-1/2}(\partial D)$ find $p\in H^{1}(D)$ and $p^s\in H^1_{\rm loc}(\mathbb{R}^m\setminus \overline{D})$ such that
\begin{subequations}
	\begin{eqnarray}
		\Delta p^s+k^2 p^s=0 \mbox{ in } \mathbb{R}^m\setminus \overline{D}, &\;& \Delta p+k^2 p=0 \mbox{ in } D,\quad p-p^s= \varphi   \mbox{ on }\partial D,\\
		{\partial_{\bfnu} p- \partial_{\bfnu} p^s-\sigma p=\psi}  \mbox{ on }\Gamma, &\;&
		\partial_{\bfnu} p -\partial_{\bfnu} p^s=\psi   \mbox{ on }\partial D\setminus \overline{\Gamma},\quad 
		\\
		\lim_{r\to\infty}r^{\frac{m-1}{2}}\left(\partial_r p^s -ikp^s\right)=0 && \mbox{ uniformly in }\hbfx=\bfx/r.
	\end{eqnarray}\label{transmission}
\end{subequations}

Indeed, if we take $\varphi:=u^i|_{\partial D}$ and $\psi:=\partial_{\bfnu} u^i|_{\partial D}$,  then the solution $u$ of problem (\ref{Forward_problem})
is given by $u:=p$ in $D$ and $u:=p^s+u^i$ in  $\mathbb{R}^m\setminus \overline{D}$. In addition, since $u$ is continuous across $\partial D$,  we have $u \in H^1_{\rm loc}(\mathbb{R}^m)$.  

Now we define the bounded linear operator ${\mathcal H}: L^2(\mathbb{S})\to H^{1/2}(\partial D)\times H^{-1/2}(\partial D)$ by
\begin{equation}
	{\mathcal H}(g):=\left(w_g^{(\lambda)}|_{\partial D},\; {\partial_{\bfnu} w_g^{(\lambda)}}|_{\partial D}\right), \label{bigH}
\end{equation}
where $w_g^{(\lambda)}$  is the total field  in the scattering problem (\ref{auxiliary_problem}) with the Herglotz wave function incident field $u^i:=u^i_g$; see (\ref{herg}) for the definition of $u^i_g$.  

Let  $\overline{\mbox{R}({\mathcal H})}$ denote the closure of the range of  ${\mathcal H}$ in $H^{1/2}(\partial D)\times H^{-1/2}(\partial D)$, and define  the bounded linear and  compact operator ${\mathcal G}: \overline{\mbox{R}({\mathcal H})}\to L^2(\mathbb{S})$ by
\begin{equation}\label{defG}
	{\mathcal G}: \, (\varphi,\psi)\mapsto p_\infty,
\end{equation}
where $p_\infty$ is the far field pattern of $p^s$ that solves (\ref{transmission}) ({the compactness of  ${\mathcal G}$ follows from  real analyticity of  $p_\infty(\hbfx)$}).  The discussion in  \Cref{sec:DAIP} implies the following factorization for the modified far field operator
$${{\mathcal F}={\mathcal G}{\mathcal H}.}$$
\noindent
{For any $h\in H^1(D)$ that is solution to the Helmholtz equation $\Delta h+k^2h=0$ let us define  $v\in H^1(D)$ by}
\begin{equation}\label{vpart}
	v(\bfx):=h(\bfx) {-}\int_{\Gamma}\sigma(\bfy)h(\bfy)\Phi(\bfx,\bfy)\,ds_y,
\end{equation}
which is a solution to the {Helmholtz equation in $D$}.  
Define, for $\bfx{\not \in} \Gamma$,
\begin{eqnarray*}
	p(\bfx)&:=&h(\bfx) =\displaystyle{\int_{\partial D}\left({\partial_{\bfnu} v(\bfy)}\Phi(\bfx,\bfy)-{\partial_{\bfnu}  \Phi(\bfx,\bfy)}v(\bfy)\right)\, ds_y} \\&&\quad{+}\displaystyle{\int_{\Gamma}\sigma(\bfy)h(\bfy)\Phi(\bfx,\bfy)\,ds_y}, \; {\bfx}\in D,\\
	p^s(\bfx)&:=&\displaystyle{\int_{\partial D}\left({\partial_{\bfnu} \Phi(\bfx,\bfy)}v(\bfy)- {\partial_{\bfnu}  v(\bfy)}\Phi(\bfx, \bfz)\right)\, ds_y }\\&&\quad{+}\displaystyle{\int_{\Gamma}\sigma(\bfy)h(\bfy)\Phi(\bfx,\bfy)\,ds_y}, \; \bfx\in  {\mathbb R}^m\setminus{D}.
\end{eqnarray*}
Standard results in scattering theory, particularly the jump relations for single and double layer potentials in \cite[Theorem 3.1]{CK2019}, give the following result:
\begin{lemma}\label{4}
	$(p(\bm x), p^s(\bm x))$ is the solution of (\ref{transmission}) with {$\varphi:= v|_{\partial D}$ and $\psi:= \partial_{\bfnu} v|_{\partial D}$.}
\end{lemma}

In addition, from the statement of the above lemma together with  Rellich's Lemma and the definition (\ref{defG}) of the operator ${\mathcal G}$ we have:
\begin{lemma}\label{Steklov2}
	Assume that $\lambda$ is an eigenvalue of the mixed Steklov eigenvalue problem (\ref{Steklov_eig}). Let  $w_v^{(\lambda),s}$ be the unique solution of
	\begin{subequations}
		\begin{eqnarray}
			\Delta w^{(\lambda),s}+k^2 w^{(\lambda),s} &=&0 \textup{ in } \mathbb{R}^m\setminus \overline{D}, \\
			\partial_{\bfnu} w^{(\lambda),s} &=& -\partial_{\bfnu} v \textup{ on }\partial D\setminus \Gamma,\\ 
			\partial_{\bfnu} w^{(\lambda),s}+\lambda w^{(\lambda),s}&=&-\partial_{\bfnu} v-\lambda v \textup{ on }\Gamma, \\
			\lim_{r\to\infty}r^{\frac{m-1}{2}}\left(\partial_r w^{(\lambda),s} -ikw^{(\lambda),s}\right)&=&0\textup{ uniformly in }\hbfx=\bfx/r, \;r=\Vert\bfx\Vert.
		\end{eqnarray}\label{auxiliary_problem2}
	\end{subequations} 
	Then
	$${\mathcal G}(\varphi_v,\psi_v)=0, \;\;\; \textup{with} \; \varphi_v:=(w_v^{(\lambda),s}+v)|_{\partial D}, \quad \psi_v:=\partial_{\bfnu} (w_v^{(\lambda),s}+v)|_{\partial D}$$
	{where $v$ is defined by (\ref{vpart})  with  $h:=h^{(\lambda)}$, the eigenfunction  corresponding to the eigenvalue $\lambda$}.
\end{lemma}

\begin{remark}
	{	It can be seen from the variational formula of the problem (\ref{Steklov_eig}) along with Assumption \ref{assum} that a mixed Steklov eigenvalue $\lambda$ must have nonnegative imaginary part. This guarantees the well-posedness of  problem (\ref{auxiliary_problem2}) under  the assumptions of Lemma \ref{Steklov2}.} 
\end{remark}

Note that $(\varphi_v,\psi_v)\in \overline{\mbox{R}({\mathcal H})}$. This follows from the fact that the set of Herglotz wave functions $u^i_g$ given by (\ref{herg}) is dense in the space
$$\left\{v\in H^1(D): \, \Delta v+k^2v=0 \qquad \mbox{in} \; D\right\}.$$
which is proven in \cite[Lemma 2.1]{CCH16}, together with the trace theorem and the well-posedness of (\ref{auxiliary_problem2}). 

The above reasoning makes it clear that if $\lambda$ is a modified mixed Steklov eigenvalue,  then the kernel of ${\mathcal G}$ is nontrivial but not necessarily the kernel of the modified far field operator ${\mathcal F}$ since the function $v$ in (\ref{vpart}) is not necessarily a Herglotz wave function. Nevertheless we can exploit the above relationships to determine the eigenvalues $\lambda$ from a knowledge of the modified far field operator. 
To this end, let $\Phi_\infty(\hbfx, \bfz)$ be the far field pattern of  the radiating fundamental solution $\Phi(x,z)$ to the Helmholtz equation.

\begin{lemma}\label{g1}
	Assume that $\lambda\in {\mathbb C}$ is not an eigenvalue of (\ref{Steklov_eig}). Then $\Phi_\infty(\cdot, \bm z)\in \mbox{\em R}({\mathcal G})$ for any $\bfz\in D$.
\end{lemma}
\begin{proof}
	Let $\bfz\in D$ and, since $\lambda\in {\mathbb C}$ is not an eigenvalue of (\ref{Steklov_eig}), consider the unique solution $h_z\in H^1(D)$ to the following problem
	\begin{eqnarray*}
		&\Delta h_z+k^2h_z=0 &\mbox{ in }D,\\
		&\partial_{\bfnu} (h_z-\Phi(\cdot, \bfz))-\sigma h_z=-\lambda (h_z-\Phi(\cdot, \bfz)) & \mbox{ on } \Gamma, \\
		& \partial_{\bfnu} (h_z-\Phi(\cdot, \bfz))=0 &\mbox{ on }\partial D\setminus \Gamma.
	\end{eqnarray*}
	Now define $v_z\in H^1(D)$ by
	\begin{equation}\label{vpart1}
		v_z(x):=h_z(x)+\int_{\Gamma}\sigma(\bfy)h_z(y)\Phi(\bfx,\bfy)\,ds_y,
	\end{equation}
	which is a solution to the Helmholtz equation in $D$.  {Then, by construction using Lemma \ref{4},}
	$${\mathcal G}(\varphi_z,\psi_z)=\Phi_\infty(\cdot ,\bm z) , \;\;\; \mbox{with} \; \varphi_z:=(w_z^{(\lambda),s}+v_z)|_{\partial D}, \quad \psi_z:=\partial_{\bfnu} (w_z^{(\lambda),s}+v_z)|_{\partial D},$$
	where $w_z^{(\lambda),s}$ is the solution of (\ref{auxiliary_problem2}) with $v:=v_z$. {This can be deduced in the same way as in the proof of Lemma \ref{Steklov2}, noting that now $h$ in the definition of $v$ is not a mixed Steklov eigenfunction, but rather the solution of the mixed Steklov problem with data coming from the boundary traces of the fundamental solution $\Phi(\bfx,\bm z)$ for $\bm z \in D$. Hence the solution $p^s$ to (\ref{transmission}) corresponding to $(\varphi_z,\psi_z)$ is identically equal to $\Phi(.,\bm z)$.}
\end{proof}
\begin{lemma}\label{g2}
	Assume that $\lambda\in {\mathbb C}$ is an eigenvalue of (\ref{Steklov_eig}). Then the set of points $\bfz\in D$ for which $\Phi_\infty(\cdot, z)\in \mbox{\em R}({\mathcal G})$  is nowhere dense in $D$.
\end{lemma}
\begin{proof}
	Assume to the contrary that $\Phi_\infty(\cdot, \bfz) \in \mbox{R}({\mathcal G})$  for $\bfz$ in a dense subset of a ball $B$ included in $D$. Then there exists  $(\varphi_z,\psi_z) \in  \overline{\mbox{R}({\mathcal H})}$ such that ${\mathcal G}(\varphi_z,\psi_z)=\Phi_\infty(\cdot, \bfz)$. 
	We call $p_z$ and $p_z^s$ the solution of  (\ref{transmission})  with $\varphi:=\varphi_z$ and $\psi:=\psi_z$.

	Following similar arguments as in the proof of \cite[Lemma 2.1]{CCH16} together with Lemma \ref{Steklov2}, one obtains that if  $\lambda$ is an eigenvalue of (\ref{Steklov_eig}), {then a pair $(\varphi,\psi)\in \overline{\mbox{R}({\mathcal H})}$ is such that}  
	$$\varphi:=(w_v^{(\lambda),s}+v)|_{\partial D}, \quad \psi:=\partial_{\bfnu} (w_v^{(\lambda),s}+v)|_{\partial D},$$
	where $w_v^{(\lambda),s}$ solves (\ref{auxiliary_problem2}) with some incident wave $v\in H_{\rm{inc}}$ where
	\begin{equation}\label{Hinc}
		H_{\rm{inc}}:=\left\{v\in H^1(D): \, \Delta v+k^2v=0\right\}.
	\end{equation}    
	Thus for the pair $(\varphi_z,\psi_z)$ we have the corresponding $v_z\in H_{\rm{inc}}$ and $w_z^{(\lambda),s}$ solving (\ref{auxiliary_problem2}) with $v:=v_z$. An application of Rellich's Lemma implies  that  $p_z=\Phi(\cdot, \bfz)$ in ${\mathbb R}^m\setminus\overline {D}$ and thus $p_z$ satisfies
	\begin{eqnarray*}
		\Delta p_z+k^2p_z&=&0 \mbox{ in }D,\\
		\partial_{\bfnu} p_z -\sigma p_z+\lambda p_z&=&\partial_{\bfnu} \Phi(\cdot, \bfz)+\lambda \Phi(\cdot, \bfz)  \mbox{ on } \Gamma, \\
		\partial_{\bfnu} p_z & =& \partial_{\bfnu} \Phi(\cdot, \bfz)   \mbox{ on }\partial D\setminus \Gamma.
	\end{eqnarray*}
	From the Fredholm alternative, the above problem is solvable if and only if
	\begin{equation}\label{ort}
		\int_{\partial D}{\partial_{\bfnu} \Phi(\cdot, \bfz)}\overline{p}_z^{(\lambda)}\, ds+ \int_{\Gamma} \lambda \Phi(\cdot, \bfz) \overline{p}_z^{(\lambda)}\, ds=0, \qquad \mbox{for all } z\in B,
	\end{equation}
	where $\overline{p}_z^{(\lambda)}$ is in the kernel of the adjoint problem, i.e. satisfies
	\begin{eqnarray*}
		\Delta  {p}_z^{(\lambda)}+k^2 {p}_z^{(\lambda)}&=&0 \mbox{ in }D, \\
		\partial_{\bfnu}  {p}_z^{(\lambda)} -\overline{\sigma}  {p}_z^{(\lambda)}+\overline{\lambda} {p}_z^{(\lambda)}&=&0   \mbox{ on } \Gamma, \\ \partial_{\bfnu}  {p}_z^{(\lambda)}&=&0 \mbox{ on }\partial D\setminus \Gamma.
	\end{eqnarray*}
	Using the boundary conditions for ${p}_z^{(\lambda)}$, (\ref{ort}) gives
	$$\int_{\partial D}\left({\partial_{\bfnu} \Phi(\cdot, \bfz)}\overline{p}_z^{(\lambda)}-{\partial_{\bfnu}  \overline{p}_z^{(\lambda)}}\Phi(\cdot, \bfz)\right)\, ds+ \int_\Gamma \sigma \overline{p}_z^{(\lambda)}\Phi(\cdot, \bfz)\,ds =0,$$
	which holds now for all $\bfz\in D$ since the left hand side as a function of $z$ is a solution to the Helmholtz equation in $D$ hence real analytic. Thus we have that 
	$$v(z):=\overline{p}_z^{(\lambda)}{-}\int_\Gamma \sigma \overline{p}_z^{(\lambda)}\Phi(\cdot, \bfz)\,ds =0.$$
	Similarly to the discussion at the beginning of this section,  $p$ and $p^s$ defined by the formula below  (\ref{vpart1})  corresponding to this $v:=0$, solve (\ref{transmission} )with zero data, hence both are zero. Thus the eigenfunction $\overline{p}_z^{(\lambda)}=0$,  which is a contradiction.
\end{proof}

Now we are ready to show how the eigenvalues, real or complex, of the mixed Steklov eigenvalue problem (\ref{Steklov_eig}) can be determined from the modified far field operator $\cF$, by means of the linear sampling method. 

To this end,  we consider the {\it modified far field equation} given by
\begin{equation}\label{mfe}
	(\cF g)(\hbfx)=\Phi_\infty(\hbfx, \bfz) \qquad \mbox{for $\bfz\in D$},
\end{equation} 
and let $v_g$ denote the Herglotz wave function with kernel $g\in L^2({\mathbb S})$ 
$$v_g(\bfx) :=\int_{\mathbb{S}}\exp(ik\bfx\cdot\bfd)g(\bfd)\,ds_d.$$
\begin{theorem} \label{lsm} Let $\{g_n^z\}_{n=0}^{\infty}$ denote a sequence in $L^2({\mathbb S})$ such that 
	\begin{equation}\label{app}
		\lim\limits_{n \to \infty}\|\cF g_n^z-\Phi_\infty(\hbfx, \bfz)\|_{L^2({\mathbb S})}=0.
	\end{equation}
	\begin{enumerate}   
		\item Assume  $\lambda$ is not an eigenvalue of (\ref{Steklov_eig}). Then, for  every $\bfz\in D$ there exists a sequence $\{g_n^z\}$ in  $L^2({\mathbb S})$ satisfying (\ref{app})
		such that $\|v_{g_n^z}\|_ {H^{1}(D)}$ is bounded.
		\item Assume  $\lambda$ is an eigenvalue of (\ref{Steklov_eig}). Then for any sequence $\{g_n^z\}$ in  $L^2({\mathbb S})$ satisfying (\ref{app}), $\|v_{g_n^z}\|_ {H^{1}(D)}$ cannot be bounded for all $\bfz$ dense in a ball $B\subset D$.  
	\end{enumerate}
\end{theorem}
\begin{proof}
	If $\lambda$ is not an eigenvalue of (\ref{Steklov_eig}), then from \Cref{g1} we have that ${\mathcal G}(\varphi_z,\psi_z)=\Phi_\infty(\cdot ,\bfz)$ with $(\varphi_z,\psi_z)\in \overline{\mbox{R}({\mathcal H})}$. Thus, there exits a sequence  $\{g_n^z\}$  in  $L^2({\mathbb S})$ such that $\left(w_{g_n^z}^{(\lambda)}|_{\partial D},\; \partial_{\bfnu} w_{g_n^z}^{(\lambda)}|_{\partial D}\right)$  converges to $(\varphi_z,\psi_z)$ in $H^{1/2}(\partial D)\times H^{1/2}(\partial D)$, where $w_{g_n}^{(\lambda)}$  is the total field  in the scattering problem (\ref{auxiliary_problem})  with the Herglotz wave function $v_{g_n^z}$ as  incident field. In particular, $v_{g_n^z}$ converges to a $v^z\in H_{\rm{inc}}$ given by (\ref{Hinc}). Hence, by continuity of ${\mathcal G}$ we have 
	$$\lim\limits_{n \to \infty}\|\cF g_n^z-\Phi_\infty(\cdot, \bfz)\|_{L^2({\mathbb S})}= \lim\limits_{n \to \infty}\|{\mathcal G}({\mathcal H}g_n^z)-\Phi_\infty(\cdot, \bfz)\|_{L^2({\mathbb S})}=0,$$
	and 
	$$ \lim_{n \to \infty}v_{g_n^z}=v^z \quad \mbox{in}\; H^1(D),$$
	which proves part {\it 1}.

	Next, assume to the contrary that $\|v_{g_n^z}\|_ {H^{1}(D)}$, with $g_n^z$ satisfying (\ref{app}), is bounded for all $z\in B$ in a ball included in $D$. This means that  $\{v_{g_n^z}\}$ converges weakly to some $v_z\in H_{\textup{inc}}$.  Let 
	$$\varphi_z:=(w_z^{(\lambda),s}+v_z)|_{\partial D}, \quad \psi_z:= \partial_{\bfnu} (w_z^{(\lambda),s}+v_z)|_{\partial D},$$
	where $w_z^{(\lambda),s}$ solves (\ref{auxiliary_problem2} )with $v:=v_z$, which {is} the weak limit of ${\mathcal H}g_n^z$ (note that the trace operator is continuous and the  solution of (\ref{auxiliary_problem2})  depends continuously of $v$). Since ${\mathcal G}$ is compact we have that ${\mathcal G}{\mathcal H}g_n^z$ converge strongly to 
	${\mathcal G}(\varphi_z, \psi_z)$ and since ${\mathcal F}g_n^z={\mathcal G}{\mathcal H}g_n^z$  we conclude that ${\mathcal G}(\varphi_z, \psi_z)=\Phi_\infty(\cdot, \bfz)$ for all $\bfz\in B\subset D$. But this is not possible from  \Cref{g2}, and the second part is proven.
\end{proof}

The above theorem states that  for every  sequence $\{g_n^z\}$ in  $L^2({\mathbb S})$ satisfying (\ref{app})  we have that $\|v_{g_n^z}\|_ {H^{1}(D)}$ is bounded  for all $z\in B\subset D$ if and only if   $\lambda$ is not an eigenvalue of (\ref{Steklov_eig}). This provides the mathematical foundation of the {\it linear sampling method} for determining these eigenvalues from $\cF$. We remark that for the existence of a sequence  $\{g_n^z\}$ in  $L^2({\mathbb S})$  satisfying (\ref{app}),  we need that $\Phi_\infty(\hbfx, \bfz)$ is in the closure of the range of $\cF$. This is guaranteed  if $\lambda$ is not an eigenvalue of (\ref{Steklov_eig}) since the range of $\cF$ is dense in $L^2({\mathbb S})$.  If $\lambda$ is an eigenvalue of (\ref{Steklov_eig}) this information is not available and may depend on $\Gamma$ and $\sigma$. Thus, if $\lambda$ is an eigenvalue of (\ref{Steklov_eig}), either such a sequence $\{g_n^z\}$ exists and  $\|v_{g_n^z}\|_ {H^{1}(D)}$ becomes unbounded, or such   $\{g_n^z\}$  does not exist in contrast to the behavior when  $\lambda$ is not an eigenvalue of (\ref{Steklov_eig}). Either way the behavior differs from when $\lambda$ is not an eigenvalue, and this fact can potentially be used in the detection of  eigenvalues.

Due to {the} ill-posedness of (\ref{mfe}), one actually solves the Tikhonov regularized version of the modified far field equation
$$ \left(\alpha I+\cF^*\cF\right)g=\cF^*\Phi_\infty(\cdot, \bfz) \qquad \mbox{for $\bfz\in D$}.$$

The problem with the above {criterion}  of determining the eigenvalues $\lambda$ is that there is no theoretical guarantee that the regularized solution $g_\alpha$ of the modified far field equation would capture the behavior of $v_g$ described by  \Cref{lsm}. Nevertheless, as our numerical examples show, this seems to be the case in practice.

The {\it {generalized} linear sampling method}, first introduced in \cite{AH2014} aims at correcting for the above theoretical deficiency of the linear sampling method. We shall present {an} alternative {criterion} to determine the eigenvalues of (\ref{Steklov_eig}) from $\cF$ based on this method which unfortunately requires some restrictive assumptions. In this framework, we have the modified far field operator ${\cF}: L^2(\mathbb{S})\to L^2(\mathbb{S})$, which assumes the factorization ${\mathcal F}={\mathcal G}{\mathcal H}$ with  ${\mathcal H}:L^2({\mathbb S})\to H^{1/2}(\partial D)\times H^{-1/2}(\partial D)$ and ${\mathcal G}:\overline{\mbox{R}(\mathcal H)}\to L^2(\mathbb{S})$. In addition let  ${\mathcal B}: L^2(\mathbb{S})\to {\mathbb R}^+$  be a continuous functional. For a given parameter $\alpha>0$ and $\phi_z:=\Phi_\infty(\cdot, \bfz)\in L^2({\mathbb S})$ we  define the  functional 
$$J_\alpha(\phi_z,g):=\alpha {\mathcal B}(g) + \left\|{\mathcal F}g-\phi_z\right\|_{L^2({\mathbb S})}^2$$
and let
\begin{equation}\label{func1}
	j_\alpha(\phi_z):=\inf\limits_{g\in L^2({\mathbb S})}J_\alpha(\phi_z,g).
\end{equation}
The  proposition below provides a characterization of the range of ${\mathcal G}$ in terms ${\mathcal F}$ and ${\mathcal B}$. The following version of the  {generalized}  linear sampling method is proven in \cite[Appendix]{ACH2017}.                        
\begin{prop}\label{GLSM}
	Assume that ${\mathcal F}$ has dense range, and for any given sequence $\{g_n\}\in L^2({\mathbb S})$ the sequence $\{{\mathcal B}(g_n)\}$ is bounded if and only if the sequence $\{\left\|{\mathcal H}g_n\right\|\}$ is bounded. Let $\{g_\alpha\}$ be a minimizing sequence   such that
	$$J_\alpha(\phi_z,g_\alpha)\leq j_\alpha(\phi_z)+C\alpha$$
	with a constant  $C>0$ independent of $\alpha$. 
	Then $\phi_z \in \mbox{R}({\mathcal G})$ if and only if the sequence $B(g_\alpha)$ is bounded as $\alpha \to 0$.
\end{prop}

We notice that \Cref{g1,g2} state that  $\lambda$ is an eigenvalue of (\ref{Steklov_eig}) if and only points  $\bfz$ such that  $\Phi_\infty(\cdot, \bfz)\in \mbox{R}({\mathcal G})$ are nowhere dense in $D$.  It remains to find an appropriate continuous {functional} ${\mathcal B}$ given in terms of data that satisfies the assumption in Proposition \ref{GLSM}.

The first {criterion} uses ${\mathcal B}g:= (F^{(\lambda)}g,g)_{L^2(\mathbb{S})} $ where  $F^{(\lambda)}:L^2(\mathbb{S})\to L^2(\mathbb{S})$ is  the far field operator  corresponding to the auxiliary scattering problem (\ref{auxiliary_problem2}).   In \cite[Lemma 1(i)]{ACH2017} it is shown that this ${\mathcal B}$ satisfies the assumption of  Proposition \ref{GLSM}, provided that a real value of $\lambda$ is not eigenvalue of the  mixed Steklov eigenvalue problem: 
\begin{equation}\label{Se}
	\Delta q+k^2 q=0 \quad \mbox{in } D,  \qquad \partial_{\bfnu} q +\lambda q=0 \quad \mbox{on }\Gamma, \qquad \partial_{\bfnu} q=0 \quad \mbox{on } \partial  D\setminus{\Gamma}.
\end{equation}
We remark that the proof in \cite[Lemma 1(i)]{ACH2017} is for the case when the impedance condition is on the entire boundary, i.e. $\Gamma=\partial D$. However all the necessary mathematical tools related to the properties of $F^{(\lambda)}$ developed in \cite{CCH2014} and \cite{KG2008} simply require  that $\lambda$ is in $L^\infty(\partial D)$ (including being zero on part of $\partial D$). To avoid repetition, in the following we only state the result.
\begin{theorem}\label{deterlambda1}
	Assume that  the modified far field operator ${\cF}: L^2(\mathbb{S})\to L^2(\mathbb{S})$ has dense range, and  $\lambda$ is not an eigenvalue of (\ref{Se}).  Consider
	$$J_\alpha(\Phi_\infty(\cdot, \bfz),g):=\alpha |(F^\lambda g,g)|+\left\|{\cF}g-\Phi_\infty(\cdot, \bfz)\right\|^2,$$
	and the corresponding  $j_\alpha(\Phi_\infty(\cdot, \bfz))$ given by (\ref{func1}). Let $g_\alpha^z$ be a minimizing sequence such that 
	$$J_\alpha(\Phi_\infty(\cdot, \bfz),g^z_\alpha)\leq  j_\alpha(\Phi_\infty(\cdot, \bfz)) +C\alpha,$$
	with $C>0$ a constant independent of $\alpha>0$. Then $\lambda\in {\mathbb C}$ is a mixed modified Steklov eigenvalue of (\ref{Steklov_eig}) if and only if 
	the set of points $\bfz$ such that  $|(F^\lambda g_\alpha^z,g_\alpha^z)|<\infty$ as $\alpha \to 0$  is nowhere dense in $D$.
\end{theorem}

The assumption that $\lambda$ is not an eigenvalue of (\ref{Se}) is somewhat restrictive for $\sigma$ real valued  because it can happen that a real $\lambda$ could be simultaneously eigenvalue of  (\ref{Se})  as well a value of interest, i.e. an eigenvalue of (\ref{Steklov_eig}). {To define a second criterion, we recall that or any operator {$T: \mathcal W\to \mathcal W$}, where {$\mathcal W$} is a  complex Hilbert space with adjoint {$T^*:\mathcal W\to \mathcal W$} we define
	$$\Re(T):=\frac{1}{2}(T+T^*), \quad \Im(T):=\frac{1}{2i}(T-T^*).$$
	Then if  $F:L^2(\mathbb{S})\to L^2(\mathbb{S})$ is  the far field operator  corresponding to the physical scattering problem (\ref{Forward_problem},) we set $$F_\#:=|\Re(F)|-\Im(F).$$}
The second {criterion} uses {${\mathcal B}g:=(F_\#g,g)_{L^2(\mathbb{S})}$.}

We must show that this choice also satisfies  the assumptions of  Proposition \ref{GLSM}. For this we need to assume that $\Gamma$ is such that {there is no non-trivial Herglotz wave function} $v_g$ such that $v_g|_{\Gamma}=0$, i.e. $H$ is injective.   {This is a geometric assumption on $\Gamma$. Open surfaces or arcs that do not satisfy this assumption cannot be fully characterized. Moreover, the assumption excludes $\Gamma$ being part of a straight line or plane, or part of a circle/sphere of certain particular radius.} 

We now adopt the notation
$$H^{\demi}(\Gamma):=\{u|_{\Gamma}: u \in H^{\demi}(\partial D)\}, \qquad \tilde H^{\demi}(\Gamma):=\{u \in
H^{\demi}(\Gamma): \mbox{supp} \,u \subseteq \overline{\Gamma} \},$$
i.e. in other
words, $\tilde H^{\demi}(\Gamma)$ contains functions $u \in H^{\demi}(\Gamma)$ such that their extension by zero to the whole
boundary $\partial D$ is in $H^{\demi}(\partial D)$. It is known \cite{McLean} that  $H^{-\demi}(\Gamma)$ is the dual space of $\tilde H^{\demi}(\Gamma)$ and $\tilde H^{-\demi}(\Gamma)$ is the dual space of $H^{\demi}(\Gamma)$, and that we have the following chain of dense and compact embedings
$$\tilde H^{\demi}(\Gamma)\subset
H^{\demi}(\Gamma)\subset L^2(\Gamma)\subset \tilde
H^{-\demi}(\Gamma)\subset H^{-\demi}(\Gamma).$$
We start by factorizing $F$ as
$$F={-G\sigma H},$$ where  $H:L^2({\mathbb S})\to H^{\demi}(\Gamma)\subset \tilde H^{-\demi}(\Gamma)$ is given by
$$H:g\mapsto v_g|_{\Gamma}, \quad \mbox{with $v_g(\bfx):=\int_{\mathbb{S}}g(\bfd)e^{ik\bfx\cdot\bfd}\,ds_d$},$$
and $G:\tilde H^{-\demi}(\Gamma)\to L^2({\mathbb S})$ is given by
$$G:h\mapsto v_{\infty},$$ where $v_\infty$ is the far field pattern of $v\in H^1_{\textup{loc}}({\mathbb R}^m)$ satisfying
\begin{subequations}
	\begin{eqnarray}
		\Delta v+k^2 v=0  \mbox{ in } \mathbb{R}^m\setminus \Gamma,&\;&  \\
		\jmp{\partial_{\bm \nu} v}+\sigma v=h \mbox{ on }\Gamma, &\;&   \jmp{{v}}=0  \mbox{ on }\Gamma,\\
		\lim_{r\to\infty}r^{\frac{m-1}{2}}\left({\partial_r {v}}-ik{v}\right)=0  &&\mbox{ uniformly in }\hbfx=\bfx/r.
	\end{eqnarray}\label{Forward_problem2}
\end{subequations}

The adjoint operator $H^*:\tilde H^{-\demi}(\Gamma)\to L^2({\mathbb S})$ is given by 
$$H^*\varphi(\hat x):=\int_\Gamma \varphi(\bfy) e^{-ik \hbfx \cdot \bfy}\, ds_y.$$
Here and {in} what follows the superscript $*$ denotes the adjoint in the $H^{-\demi}(\Gamma)$-inner product. From the assumption we have that $H^*$ has dense range in $L^2({\mathbb S})$. Note that $\jmp{\partial_{\bfnu} v}\in \tilde
H^{-\demi}(\Gamma)$. On the other hand  from the jump relations for the single layer potential, we see that 
$$
{{\mathcal S}\varphi(\bfx)}:=\int_{\Gamma}\varphi(\bfy)\Phi(\bfx,\bfy)\,ds_y, \qquad \varphi\in \tilde H^{-\demi}(\Gamma),
$$
satisfies (\ref{Forward_problem2}) with $h:=(-I+\sigma S_{\Gamma})\varphi$, where $S_{\Gamma}: \tilde H^{-\demi}(\Gamma)\to  H^{\demi}(\Gamma)\subset\tilde H^{-\demi}(\Gamma)$ is the boundary integral operator $S\varphi(x):=\int_{\Gamma}\varphi(y)\Phi(x,y)\,ds_y$ restricted to $\Gamma$ (see \cite{McLean}) and $I$ is the identity operator.  Obviously ${\mathcal S}_\infty\varphi=H^*\varphi$. Thus we have the factorization
\begin{equation}\label{hg}
	H^*=-G(I-\sigma S_{\Gamma}).
\end{equation}
\begin{lemma}\label{lemG}
	Assume that $\Gamma$ is such that  there is no {non-trivial} Herglotz wave function $v_g$ such that $v_g|{\Gamma}=0$. Then $G:\tilde H^{-\demi}(\Gamma)\to L^2({\mathbb S})$  is injective with dense range. 
\end{lemma}
\begin{proof}
	{Injectivity of $G$ follows from the well-posedeness of the direct scattering problem.} The range of $G$ coincides with the range of $H^*$ since the operator $I-\sigma S_{\Gamma}: \tilde H^{-\demi}(\Gamma)\to \tilde H^{-\demi}(\Gamma)$ is {bijective}. To show this, we notice that $\sigma S_{\Gamma}: \tilde H^{-\demi}(\Gamma)\to \tilde H^{-\demi}(\Gamma)$ is compact due to compact embedding of  $H^{\demi}(\Gamma)$ into $\tilde H^{-\demi}(\Gamma)$. Hence, it suffices to prove injectivity. To this end let {$\varphi_0$} be such that $(I-\sigma S_{\Gamma})\varphi_0=0$, and consider 
	$$w(\bfx):=\int_{\Gamma}\varphi_0(\bfy)\Phi(\bfx,\bfy)\,ds_y$$
	which satisfies (\ref{Forward_problem2} )with $h=0$, thus $w=0$ in ${\mathbb R^m}\setminus \Gamma$. Again the jump relation for the normal derivative of the single layer potential implies implies that $\varphi_0=0$. Then the result follows from the fact that the range of $H^*$ is dense.
\end{proof}

As a consequence of the above lemma and the previous discussion we have
\begin{prop}\label{cor3} Assume that $\Gamma$ is such that  there is no {non-trivial} Herglotz wave function $v_g$ such that $v_g|{\Gamma}=0$. Then the operators $G^*: L^2({\mathbb S})\to \tilde H^{-\demi}(\Gamma)$, $H: L^2({\mathbb S})\to \tilde H^{-\demi}(\Gamma)$ and $H^*: \tilde H^{-\demi}(\Gamma)\to L^2({\mathbb S})$ are   injective with dense range.
\end{prop}

\noindent
Combining both above factorizations  we have
\begin{equation}\label{fac0}
	F=H^*(I-\sigma S_{\Gamma})^{-1}\sigma H=H^*\overline{\sigma}(\overline{\sigma}I-|\sigma|^2S_\Gamma)^{-1}\sigma H.
\end{equation}

\medskip
\noindent
Now we may verify the following:
\begin{lemma}\label{propop}
	Let $A:=(\overline{\sigma}I-|\sigma|^2S_\Gamma)^{-1}: \tilde H^{-\demi}(\Gamma)\to \tilde H^{-\demi}(\Gamma)$, and assume that $\Re(\sigma(\bm x))\geq \sigma_0>0$ for $\bfx\in \Gamma$. The following properties hold true:
	\begin{enumerate}
		\item[(i)] $\Im\left(\varphi, A\varphi\right)_{H^{-\demi}(\Gamma)}{\le } 0$.
		\item[(ii)] $\Re(A)=A_0+C$, where $C: \tilde H^{-\demi}(\Gamma)\to \tilde H^{-\demi}(\Gamma)$    is compact and $A_0$ satisfies
		$$ \mbox{$\left(\varphi, A_0\varphi\right)\geq c_0\|\varphi\|^2_{H^{-\demi}(\Gamma)}$ \;\;\; \; for all \;\; $\varphi \in\overline{\mbox{\em R}(H)}=\tilde H^{-\demi}(\Gamma)$}.$$      
		\item[(iii)] $A$  is injective  in $\overline{\mbox{\em R}(H)}=\tilde H^{-\demi}(\Gamma)$.                                                  
	\end{enumerate}
\end{lemma}
\begin{proof}
	We have 
	\begin{eqnarray*}
		\Im\left(\varphi, (\overline{\sigma}I-|\sigma|^2S_\Gamma)\varphi\right)_{H^{-\demi}(\Gamma)}&=&\Im(\sigma) {\|\varphi\|_{H^{-\demi}(\Gamma)}^2}-|\sigma|^2\Im\left(\varphi, S_\Gamma\varphi\right)_{H^{-\demi}(\Gamma)}\\
		&=&\Im(\sigma) {\|\varphi\|_{H^{-\demi}(\Gamma)}^2}-|\sigma|^2\Im\left<\varphi, S_\Gamma\varphi\right>_{\tilde H^{-\demi}, H^{\demi}}>0,
	\end{eqnarray*}
	where $\left<\cdot, \cdot\right>$ denotes the duality $H^{-\demi}(\Gamma), H^{\demi}(\Gamma)$. This follow from  the fact that $\Im(\sigma)\geq 0$ and the known fact  \cite[Lemma 1.14]{KG2008}  that 
	$$\Im\left<\phi, S\phi\right>_{\tilde H^{-\demi}(\partial D), H^{\demi}(\partial D)}<0 \qquad \mbox{for all $\phi \in \tilde H^{-\demi}(\phi)$ with $\phi\neq 0$},$$ 
	and that $\varphi\in H^{-\demi}(\Gamma)$ can be extended by zero to the whole $\partial D$ in $H^{-\demi}(\partial D)$.  The statement then follows form the fact that $A$ is the inverse of $(\overline{\sigma}I-|\sigma|^2S_\Gamma)$.

	Concerning the second part, we observe that  
	$$A:=(\overline{\sigma}I-|\sigma|^2S_\Gamma)^{-1}=\overline{\sigma}^{-1}(I-\sigma S_{\Gamma})^{-1}=\overline{\sigma}^{-1}(I+\sigma S_\Gamma(I-\sigma S_{\Gamma})^{-1}),$$
	where $\sigma S_\Gamma(I-\sigma S_{\Gamma})^{-1}$ is compact as product of the bounded operator $\sigma(I-\sigma S_{\Gamma})^{-1}$ and the compact operator  $S_{\Gamma}$. The latter is true due to compact embedding of  $H^{\demi}(\Gamma)$ into $\tilde H^{-\demi}(\Gamma)$. Therefore we can write $\Re(A):=A_0+C$ with {$A_0:=\frac{\Re(\sigma)}{|\sigma|^2}I$} coercive and {$C:=\Re(\frac{\sigma^2}{|\sigma|^2 }S_\Gamma(I-\sigma S_{\Gamma})^{-1})$} compact.

	Finally the last part is obvious since $A$ is invertible.
\end{proof}

\Cref{propop} along {Proposition} \ref{cor3} allows us to apply \cite[Theorem 2.31]{CCH16}  or \cite[Theorem 2.15]{KG2008} to state the so-called $F_\#$ factorization statement. To this end let us define
$$A_\#:=|\Re(A)|-\Im(A).$$
\begin{theorem}\label{thter}
	Let $A:=(\overline{\sigma}I-|\sigma|^2S_\Gamma)^{-1}: \tilde H^{-\demi}(\Gamma)\to \tilde H^{-\demi}(\Gamma)$, and assume that there is no {non-trivial} Herglotz wave function $v_g$ such that $v_g|_{\Gamma}=0$ and $\Re(\sigma)\geq \sigma_0>0$  on $\Gamma$. Then
	$$F_\#=(\sigma H)^*A_\# (\sigma H),$$
	where $A_\#:\tilde H^{-\demi}(\Gamma)\to \tilde H^{-\demi}(\Gamma)$ is self-adjoint and satisfies the coercivity property
	$$ \mbox{$\left(\varphi, A_\#\varphi\right)\geq c\|\varphi\|^2_{H^{-\demi}(\Gamma)}$ \;\;\; \; for all \;\; $\varphi \in\overline{\mbox{\em R}(H)}=\tilde H^{-\demi}(\Gamma)$}.$$  
	Moreover, $\mbox{\em R}(H^*)=\mbox{\em R}(F_\#^\demi)$.
\end{theorem}

Now we are ready to formulate a second {criterion} to determine the mixed modified Steklov eigenvalues of (\ref{Steklov_eig}). 
\begin{lemma}\label{lemp}
	{Assume that $\Gamma$ is such that  there is no {non-trivial} Herglotz wave function $v_g$ such that $v_g|{\Gamma}=0$}. Under the assumptions of \Cref{thter}, the  functional ${\mathcal B}: g\mapsto (F_\#g,g)_{L^2(\mathbb{S})}=\|F_\#^\demi g\|_{L^2(\mathbb{S})}$ satisfies the property that for any given sequence $\{g_n\}\in L^2({\mathbb S})$ the sequence $\{{\mathcal B}(g_n)\}$ is bounded if and only if the sequence $\{\left\|{\mathcal H}g_n\right\|\}$ is bounded as $n\to \infty$. 
\end{lemma} 
\begin{proof}
	Assume that ${\mathcal B}g_n$ is bounded. Then  using  the above factorization we have $${\mathcal B}g_n=(F_\#g_n,g_n)=(A_\# {\sigma}Hg_n, {\sigma  Hg_n)\geq  c\|{\sigma}Hg_n\|_{H^{-\demi}(\Gamma)}^2}$$ {where $c$ is the constant appearing in Theorem~\ref{thter}.}
	Thus $\| Hg_n\|_{H^{-\demi}(\Gamma)}$ is bounded.  
	
	{Noting that $Hg_n=v_{g_n}|_{\Gamma}$, we now show that this  implies  $\|v_{g_n}\|_{H^1(B_R)}$  is also bounded.  Assume to the contrary that $\|v_{g_n}\|_{H^1(B_R)}$ is unbounded, which means that $\|g_n\|_{L^2({\mathbb S})}$ is unbounded. Hence, by extracting a subsequence, we can assume that  $\|g_n\|_{L^2({\mathbb S})} \to \infty$ as $n\to \infty$. Now consider $g_n/\|g_n\|_{L^2({\mathbb S})}$. Again after extracting a subsequence,  this sequence converges to some function $\varphi$ weakly in $L^2({\mathbb S})$. Hence}  
	{$$\frac{1}{\| g_n\|_{L^2(\mathbb{S})}}\int_{{\mathbb S}} g_n(\hat y)e^{ikx\cdot \hat y}\, d \hat y \to  \int_{{\mathbb S}}\varphi(\hat y) e^{ikx\cdot \hat y}\, d \hat y:=v_\varphi\qquad \mbox{as } n\to \infty$$
		point-wise on $B_R$. 
		But due to the boundedness of $\|v_{g_n}\|_{H^{-\demi}(\Gamma)}$, 
		$$\frac{1}{\| g_n\|_{L^2(\mathbb{S})}}\|v_{g_n}\|_{H^{-\demi}(\Gamma)}\to  0 \qquad \mbox{as } n\to \infty,$$
		whence $v_\varphi|_{\Gamma}=0$ which contradicts the assumption of the lemma. Hence
		$\|v_{g_n}\|_{H^1(B_R)}$ is bounded as $n\to\infty$.}
	
	Now from  the well-posedness of (\ref{auxiliary_problem2}) we conclude that $\left\{{\mathcal H}g_n\right\}$ is bounded in $H^{1/2}(\partial D)\times H^{-1/2}(\partial D)$ where ${\mathcal H}g$  is given by (\ref{bigH}).

	Conversely assume that $\left\{{\mathcal H}g_n\right\}$ is bounded in $H^{1/2}(\partial D)\times H^{-1/2}(\partial D)$. Then from the Green's representation theorem we have that the scattered field  $w^{(\lambda),s}_{g_n}$ is bounded in $H^1(B_R\setminus D)$ and hence  the incident field $v_{g_n}$ is also bounded in $H^1(B_R)$, hence  by the trace theorem $\|Hg_n\|_{H^{-\demi}(\Gamma)}$ is bounded. Now 
	$${\mathcal B}g_n=(F_\#g_n,g_n)=(A_\# {\sigma}Hg_n, {\sigma} Hg_n)\leq C_{max}\|\sigma Hg_n\|_{H^{-\demi}(\Gamma)}^{2}$$
	{where $C_{max}<\infty$ depends on the bound on $A_{\#}$.}
	Thus ${\mathcal B}(g_n)$ is also bounded which completes the proof. 
\end{proof}

\Cref{lemp} allows us to apply Preposition \ref{GLSM} with ${\mathcal B}: =(F_\#g,g)_{L^2(\mathbb{S})}$ to obtain the following result on the determination of the eigenvalues of (\ref{Steklov_eig}).
\begin{theorem}\label{deterlambda2}
	Assume that  the modified far field operator ${\cF}: L^2(\mathbb{S})\to L^2(\mathbb{S})$ has dense range,  $\Re(\sigma)\geq \sigma_0>0$ on $\Gamma$, and that there is no {non-trivial} Herglotz wave function $v_g$ such that $v_g|_{\Gamma}=0$.  Consider
	$$J_\alpha(\Phi_\infty(\cdot, \bfz),g):=\alpha (F_\# g,g)+\left\|{\cF}g-\Phi_\infty(\cdot, \bfz)\right\|^2,$$
	and the corresponding  $j_\alpha(\Phi_\infty(\cdot, \bfz))$ given by (\ref{func1}). Let $g_\alpha^z$ be a minimizing sequence such that 
	$$J_\alpha(\Phi_\infty(\cdot, \bfz),g_\alpha)\leq  j_\alpha(\Phi_\infty(\cdot, \bfz)) +C\alpha,$$
	with $C>0$ a constant independent of $\alpha>0$. Then $\lambda\in {\mathbb C}$ is a mixed modified Steklov eigenvalue of (\ref{Steklov_eig}) if and only if 
	the set of points $\bfz$ such that  $(F_\# g_\alpha^z,g_\alpha^z)<\infty$ as $\alpha \to 0$  is nowhere dense in $D$.
\end{theorem}
In contrast to  \Cref{deterlambda1}, the above theorem provides a rigorous {criterion} for determining the eigenvalues under a geometric assumption on $\Gamma$ without restricting the value of the parameter $\lambda$. We also remark that it is possible to obtain the regularized version of the statement in \Cref{deterlambda1,deterlambda2} for noisy modified far field operator ${\mathcal F}$. We refer the reader to  \cite{ACH2017} and \cite{CCH16} for details on this matter. 

\section{Relations between eigenvalues and the surface parameter}\label{sec:resp}
We now turn our attention to the eigenvalue problem (\ref{Steklov_eig}), in particular understanding how its eigenvalues relate to the unknown coefficient  $\sigma$. This eigenvalue problem is a modification of the so called sloshing eigenvalue problem. Let us fix  a constant $\alpha\geq 0$ such that $k^2$ is not a mixed Robin-Neumann eigenvalue of 
\begin{eqnarray*}
	\Delta h+k^2 h&=&0 \quad \mbox{in } D,  \qquad \partial_{\bfnu} h +(\alpha-\sigma(x))h=0 \quad \mbox{on }\Gamma, \\
	\partial_{\bfnu} h&=&0 \quad \mbox{on } \partial  D\setminus{\Gamma}.
\end{eqnarray*}
{Such an $\alpha$ can always be found because  a given $k^2$ cannot be a mixed Robin-Neumann eigenvalue for all $\alpha\geq 0$. Indeed, the lefthand side  of the following variational form of  the above problem
	$$\int_D \left(\nabla h\cdot \nabla \overline\varphi -k^2h\overline\varphi\right) \, dx+\int_\Gamma(\alpha-\sigma)h\overline\varphi\, ds=0\qquad \mbox{for all} \;\; \varphi \in H^1(D).$$
	determines by means of the Riesz representation theorem an $\alpha$ parametric family of operators in $H^1(D)$ that is invertible plus compact and depends analytically on $\alpha \in {\mathbb C}$. Hence the Analytic Fredholm Theory  \cite{CK2019} can be applied. Since  $k$ and $D$ satisfy Assumption \ref{assum}, the kernel of this operator for $\alpha\in {\mathbb C}$ with $\Im(\alpha)<0$ is trivial. Hence the set of $\alpha$ for which the kernel is nontrivial is at most discrete.}

Hence we can define the operator ${\mathcal R}:L^2(\Gamma)\to L^2(\Gamma)$ mapping
$${\mathcal R}:\theta\mapsto h_{\theta}|_\Gamma$$
where $h_\theta\in H^1(D)$ is the unique solution of 
$$\int_D {\left(\nabla h_\theta\cdot \nabla \varphi -k^2h_{\theta} \varphi\right)} \, dx+\int_\Gamma(\alpha-\sigma)h_\theta\varphi\, ds=\int_{\Gamma}\theta\varphi \, ds \qquad \mbox{for all} \;\; \varphi \in H^1(D).$$
Obviously, this operator is compact, since $h_{\theta}|_\Gamma\in H^{\demi}(\Gamma)$ which is compactly embedded in $L^2(\Gamma)$. Then $\lambda$ is an eigenvalue of (\ref{Steklov_eig}) if and only if 
$$(-\lambda +\alpha){\mathcal R}\theta =\theta.$$
In other words $\frac{1}{-\lambda+\alpha}$ are the eigenvalues of the compact operator ${\mathcal R}$ and hence form {at most a countable discrete set (note that the eigenvalues $\lambda\neq \alpha$ from the assumption on $k^2$).} 

{When the coefficient  $\sigma(x)$ is real valued, the compact operator ${\mathcal R}$ is in addition self-adjoint. Thus, it has an infinite sequence of real eigenvalues   $\frac{1}{-\lambda_j+\alpha}$ , $j=1,\cdots$ which accumulate to zero. 
	Hence we conclude that  in this case  all eigenvalues of (\ref{Steklov_eig}) are real, and there exists a sequence $\left\{\lambda_j\right\}_{j\geq 1}$ of eigenvalues which accumulate to $\pm\infty$ (because the operator ${\mathcal R}$ is not sign definite  the eigenvalues can in principal accumulate to both $\pm \infty$). 
	However we show next, that they accumulate  only to $-\infty$ as $j\to \infty$ since there are only finitely many positive eigenvalues}.  
In addition the corresponding eigenfunctions $h_j$ form a Riesz basis for $H^1(D)$.  If $\Im(\sigma)>0$ there are no real eigenvalues,  and since ${\mathcal R}$ is not selfadjoint the existence of non-real complex eigenvalues requires a special investigation, which is beyond the scope of this paper.
\begin{theorem}
	{Suppose the fixed wave number $k$  and domain $D$ satisfy  Assumption \ref{assum}. For  real valued positive $\sigma>0$, there exists at least one positive eigenvalue $\lambda>0$ of (\ref{Steklov_eig}). Furthermore, there are at most finitely many positive eigenvalues of (\ref{Steklov_eig}), hence  they accumulate only at $-\infty$.}
\end{theorem}
\begin{proof}
	Assume  to the contrary that all eigenvalues $\lambda_j\leq 0$.  For the eigenfunction $h_j$ corresponding to $\lambda_j$ we have
	$$\int_D {\left(|\nabla h_j|^2 -k^2|h_j|^2\right)} \, dx-\int_\Gamma\sigma|h_j|^2\, ds=-\lambda_j\int_{\Gamma}|h_j|^2 \, ds \geq 0.$$
	Since the eigenfunctions form a Riesz basis we have that for all $h\in H^1(D)$ 
	\begin{equation}\label{ine}
		\int_D {\left(|\nabla h|^2 -k^2|h|^2\right)} \, dx-\int_\Gamma\sigma|h|^2\, ds \geq 0,
	\end{equation}
	which is not possible. Indeed if $\sigma>0$ this is not satisfied for $h=1$.

	Next assume that there exists a sequence of positive eigenvalues $\lambda_j>0$, $j\to +\infty$  converging to $+\infty$ with eigenfunctions $h_j$ normalized such that
	$$\|h_j\|_{H^1(D)}+\|h_j\|_{L^2(\Gamma)}=1.$$
	{Then since the left hand side of the following expression is bounded }
	\begin{equation}\label{ff}
		\int_D {\left(|\nabla h_j|^2 -k^2|h_j|^2\right)} \, dx-\int_\Gamma\sigma|h_j|^2\, ds=-\lambda_j\int_{\Gamma}|h_j|^2 \, ds,
	\end{equation}
	{we see that $h_j\to 0$ in $L^2(\Gamma)$ as $j\to \infty$. Next,  a subsequence $h_j$} converges weakly in $H^1(D)$ to some $h\in H^1(D)$. Since each $h_j$ satisfies  
	$$\int_D {\left(\nabla h_j \cdot \nabla \varphi -k^2 h_j \varphi \right)}\, dx-\int_\Gamma\sigma h_j \varphi \, ds+\lambda_j\int_{\Gamma} h_j \varphi \, ds=0 \qquad \forall \, \varphi \in H^1(D),$$
	we {find} that the weak limit satisfies $\Delta h+k^2h=0$ and hence $\bfnu\cdot \nabla h\in H^{-\demi}(\partial D)$ is well defined. Since $\bfnu\cdot \nabla h_j=0$ on  $\partial D\setminus \Gamma$ we can conclude that {the} weak limit also satisfies  $\bfnu\cdot \nabla h=0$ on  $\partial D\setminus \Gamma$. From the above we also {know that} $h=0$ on $\Gamma$. {Hence, using Assumption \ref{assum}}, we can conclude that the weak limit $h=0$ in $D$. Therefore , {choosing } a subsequence,  we see that $h_j\to 0$ strongly in $L^2(D)$. From (\ref{ff}) we can now conclude that, {again choosing}  a subsequence, {also $\|\nabla h_j\|_{L^2(D)}\to 0$  which contradicts the assumed  normalization and completes the proof.}
\end{proof}
\begin{remark}
	{\em Note that it is possible to prove that an eigenvalue $\lambda(k)$ of (\ref{Steklov_eig}) seen as function of the wave number $k$ blows up as $k\to k_0$ where $k_0^2$ is an eigenvalue of (\ref{mix}).}
\end{remark}

Let $\tau_1:=\tau_1(D, \Gamma, \alpha)>0$  for $0<\alpha<\infty$ be the first mixed Robin-Neumann eigenvalue of 
$$ \Delta u+\tau u=0 \quad \mbox{in} \; D, \qquad \partial_{\bfnu} u+\alpha u = 0 \quad \mbox{on}\; \Gamma, \qquad  \partial_{\bfnu} u =0\quad \mbox{on} \;  \partial D\setminus \Gamma,$$
which, since this is an eigenvalue problem for a positive self-ajoint operator, satisfies
\begin{equation}
	\tau_1=\inf\limits_{u\in H^1(D), u\neq 0}\frac{\|\nabla u\|^2_{L^2(D)}+\alpha\|u\|^2_{L^2(\Gamma)}}{\|u\|^2_{L^2(D)}}.
\end{equation}
Hence for every nonzero $h\in H^1(D)$ we have
$$\frac{1}{\tau_1}\|\nabla h\|^2_{L^2(D)}+\frac{\alpha}{\tau_1}\|h\|^2_{L^2(\Gamma)}\geq \|h\|^2_{L^2(D)}.$$
Again, $\sigma$ is real valued, we can then estimate for some $\Lambda\geq 0$,
\begin{eqnarray*}
	&\int_D {\left(|\nabla h|^2 -k^2|h|^2 \right)}\, dx-\int_\Gamma\sigma|h|^2\, ds +\Lambda \int_\Gamma|h|^2\,ds \\
	&\geq \left(1- \frac{k^2}{\tau_1}\right)\int_{D}|\nabla h|^2\,dx +\int_{\Gamma} \left(\Lambda -\frac{\alpha k^2}{\tau_1}-\sigma\right)|h|^2\,ds \\
	&\geq \left(1- \frac{k^2}{\tau_1}\right)\int_{D}|\nabla h|^2\,dx+\Lambda_*\int_{\Gamma}|h|^2\,ds.
\end{eqnarray*}
We choose $\Lambda\geq 0$  sufficiently large such that $\Lambda_*:=\min_{\Gamma}\left(\Lambda -\alpha k^2/\tau_1-\sigma\right)>0$ and $k^2<\tau_1$.  {For such a choice of $\Lambda$ and $k$}, our eigenvalue problem (\ref{Steklov_eig})  becomes 
$$\int_D|\nabla h|^2 -k^2|h|^2 \, dx-\int_\Gamma\sigma|h|^2\, ds +\Lambda \int_\Gamma|h|^2\,ds=(\Lambda-\lambda)  \int_\Gamma|h|^2\,ds,$$ where now the left-hand-side defines a self-adjoint positive operator. Therefore its eigenvalues $\Lambda_j:=\Lambda-\lambda_j$ satisfy the Courant-Fischer inf-sup principle (see e.g. \cite{lax}).

In particular, from the above assuming that $\sigma>0$ and the wave number $k^2<\tau_1(D, \Gamma, \alpha)$ (which does not depend on the unknown function $\sigma$), we have that the largest positive eigenvalue of  (\ref{Steklov_eig},) whose existence is guaranteed by  (\ref{mix}), satisfies
\begin{equation}
	\lambda_1=\sup_{u\in H^1(D), u\neq 0}\frac{-\int_D|\nabla u|^2 +k^2|u|^2 \, dx+\int_\Gamma\sigma|hu|^2\, ds}{\int_\Gamma|u|^2\, ds}.
\end{equation}
This property states that the largest positive eigenvalue is monotonic increasing with respect to the $L^\infty(\Gamma)$ norm of $\sigma$. It is possible to write similar relations for other eigenvalues based on the Courant-Fischer inf-sup principle  for all the eigenvalues $\Lambda-\lambda_j$. Furthermore the shifting constant $\Lambda$ which depends on $\sigma$ is related to the magnitude of  $\lambda_1$ is, and this can be used to get some information on  $\sigma$. We also noticed that if  $\sigma<0$ in $\Gamma$  such that {$-\sigma>\alpha k^2/\tau_1$} and for $k^2<\tau_1(D, \Gamma, \alpha)$ all the eigenvalues of (\ref{Steklov_eig})  are negative.

The simplest case occurs when $\sigma$ is constant.  In this case, if we let $\lambda_n^{(0)}$ denote the $n$-th eigenvalue of (\ref{Steklov_eig}) with $\sigma=0$ we immediately have
\begin{equation}
	\lambda_n=\sigma+\lambda_n^{(0)}, \;n=1,2,\cdots.
	\label{simple}
\end{equation}
In this case we may determine $\sigma$ by measuring a single $\lambda_n$ and computing $\lambda_n^{(0)}$.

Now if  $\sigma$ is a small perturbation of a constant $\sigma_0$ we could try a perturbation expansion.  Let
\begin{eqnarray*}
	\sigma&=&\sigma_0+\epsilon \sigma_1,\\
	h&=&h_0+\epsilon h_1+\epsilon^2 h_2+\cdots,\\
	\lambda&=&\tilde{\lambda}_0+\epsilon\tilde{\lambda}_1+\epsilon^2\tilde{\lambda}_2+\cdots,
\end{eqnarray*}
where $\epsilon$ is a small parameter.
We are particularly interested in the case where $\sigma_0$ is constant, but at this stage allow $\sigma_0$ and $\sigma_1$ to vary long $\Gamma$.  Plugging in to (\ref{Steklov_eig}) and equating orders of $\epsilon$ we obtain to $O(1)$
\begin{eqnarray*}
	\Delta h_0+k^2h_0&=&0  \mbox{ in }D, \quad  \partial_{\bfnu} h_0 -\sigma_0h_0=-\tilde{\lambda}_0 h_0  \mbox{ on } \Gamma, \\ 
	\partial_{\bfnu} h_0 &=&0  \mbox{ on }\partial D\setminus \Gamma,
\end{eqnarray*}
which is our standard Steklov eigenvalue problem.  To $O(\epsilon)$ we obtain
\begin{eqnarray*}
	\Delta h_1+k^2h_1&=&0 \mbox{ in }D, \quad \partial_{\bfnu} h_1-\sigma_0h_1{+\tilde{\lambda}_0h_1}=\sigma_1h_0-\tilde{\lambda}_1h_0  \mbox{ on } \Gamma, \\
	\partial_{\bfnu} h_1 &=&0 \mbox{ on }\partial D\setminus \Gamma.
\end{eqnarray*}
This is not generally solvable, and the solvability criterion furnishes a way to determine the solution. In particular the $\sigma_1h_0-\tilde{\lambda}_1h_0$ must be orthogonal to the eigenspace of the adjoint of the eigenvalue  problem (\ref{Steklov_eig}). This relation, which involves the eigenfunctions (\ref{Steklov_eig}) of  could be used to compute $\sigma_1$. This perturbation approach can be rigorously justified for real or complex $\sigma$ using perturbation analysis with  perturbation theory with explicit first correction term of Osborn \cite{osborn}.

\section{Numerical experiments}\label{sec:num}
We now show some numerical results that support our claims in the paper, for the sake of simplicity, we restrict our attention to examples in $\mathbb{R}^{2}$ instead of $\mathbb{R}^{3}$. 

We start by investigating the sensitivity of  eigenvalues of (\ref{Steklov_eig}) to the size, position and mixed boundary conditions of the scatter. 
Here $D$ is a half circle or  a quadrant, see \Cref{Domains_Stoklev_eigenvalue_problem},  the wave number $k=2$, the function $\sigma$ depends on the angle $\theta$ and is given by: 
\begin{eqnarray*}
	\sigma(\theta) = \left\{\begin{array}{cl}
		1& \textup{if } 0<\theta<\frac{\pi}{3}+\beta,\\[1em]
		1+\alpha\sin^2(3\beta-3\theta)& \textup{if } \frac{\pi}{3}+\beta<\theta<\frac{2\pi}{3}+\beta,\\[1em]
		1&\textup{if } \frac{2\pi}{3}+\beta<\theta<\pi,
	\end{array}\right.
\end{eqnarray*}
where $\theta$ is the angle of $(x,y)\in \Gamma$, $\alpha\ge 0$ and $\beta\in [0, \pi/3]$ are  parameters.
\begin{figure}[h!]
	\centerline{
		\hbox{\includegraphics[height=1in]{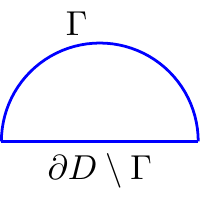}}\hspace{3cm}
		\hbox{\includegraphics[height=1in]{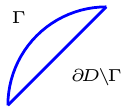}}}
	\caption{The domains of the Steklov eigenvalue problem (\ref{Steklov_eig}).}
	\label{Domains_Stoklev_eigenvalue_problem}
	\centering
\end{figure}

We use a a standard conforming finite element method to find approximations of the largest Steklov eigenvalues.  We  assume that the mesh is chosen so that each component of the boundary  $\Gamma$ and $\partial D\setminus\Gamma $ is exactly covered by a union of edges; see \Cref{mesh} for an example.

\begin{figure}
	\begin{center}
		\begin{tabular}{cc}
			\resizebox{0.47\textwidth}{!}{\includegraphics{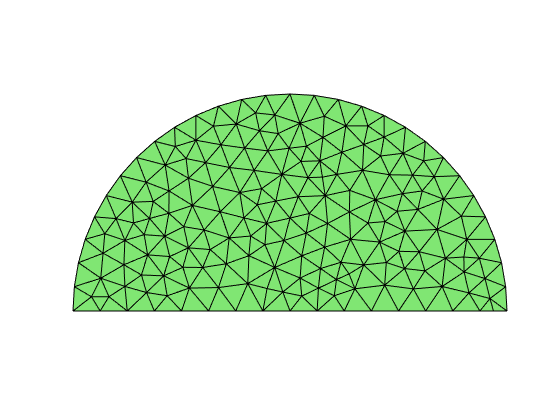}}&
			\resizebox{0.47\textwidth}{!}{\includegraphics{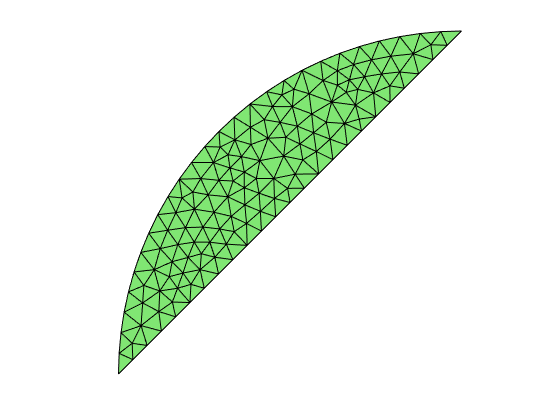}}\\
		\end{tabular}
		\caption{One level  mesh of the domains in \Cref{Domains_Stoklev_eigenvalue_problem}.}
		\label{mesh}
	\end{center}
	
\end{figure}

Let $V_{h}$ denote the space of continuous piecewise quadratic finite elements on a mesh $\mathcal{T}_{h}$ of $\Omega$ consisting of regular triangles of maximum diameter $h>0$.  The computational domain is then $\Omega_{h}=\cup_{K \in \mathcal{T}_{h}} K$. Then we find that the discrete eigenpairs $(\lambda_h, w_h)$ satisfy
$$
\left(\nabla w_h, \nabla v_h\right)-k^{2}\left(w_h, v_h\right)-\langle  \sigma w_h, v_h\rangle_{\Gamma}=-\langle \lambda_h w_h,  v_h\rangle_{\Gamma} \quad \forall v_h \in V_{h}.
$$
We only compute the first six eigenvalues $\lambda_{hj}, j=1, \ldots, 6$ since we can only determine the first few eigenvalues from the far field data. 
We shall investigate how changes in $\sigma$ (due to changes in $\alpha$ and $\beta$) are detected by changes in the measured Steklov eigenvalues.  When reporting changes in eigenvalues we use the relative change  defined by
\begin{equation}
	\frac{\lambda_{hj}\left(\alpha,\beta\right)-\lambda_{hj}(0,0)}{\lambda_{hj}(0,0)}, \quad j=1, \ldots, 6.\label{eig_cha}
\end{equation}
\begin{figure}
	\begin{center}
		\begin{tabular}{cc}
			\resizebox{0.47\textwidth}{!}{\includegraphics{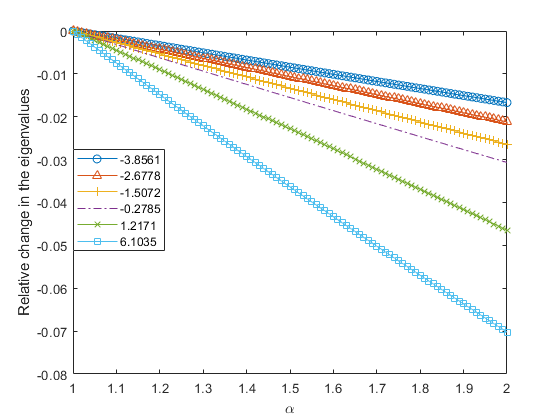}}&
			\resizebox{0.47\textwidth}{!}{\includegraphics{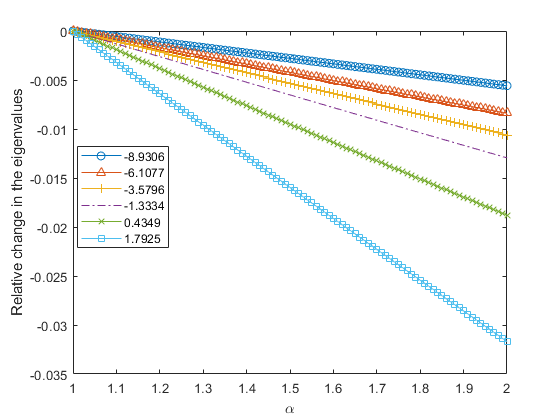}}\\
			\resizebox{0.47\textwidth}{!}{\includegraphics{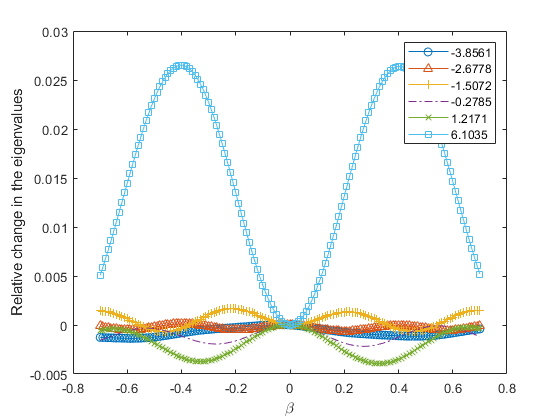}}&
			\resizebox{0.47\textwidth}{!}{\includegraphics{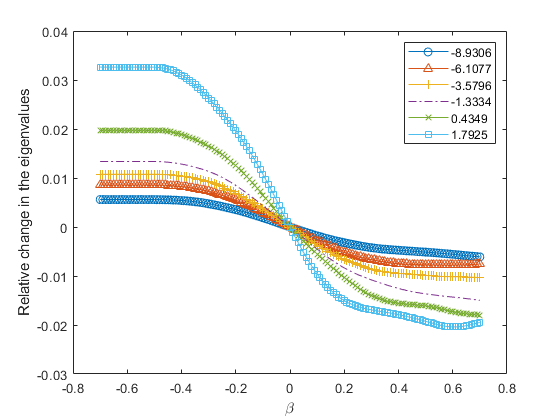}}\\
		\end{tabular}
		\caption{Sensitivity of Steklov eigenvalues to changes in the parameters $\alpha$ and $\beta$,  for two different screens. We plot the relative change in the first six Steklov
			eigenvalues for each  screen. The numbers in
			the legends refer to the computed eigenvalues of the unperturbed screen with $\sigma=1$. The left column is for the half circle and right is for the quadrant. Top is the change of $\alpha$ from $1$ to $2$ with $\beta=0$, bottom is the change of $\beta$ from $-0.7$ to $0.7$ with $\alpha=1$.}
		\label{eigchange}
	\end{center}
	
\end{figure}

\Cref{eigchange} shows plots of the relative change (defined by (\ref{eig_cha}))  in the first six eigenvalues computed by our finite element model as a function of $\alpha$ (top row) and $\beta$ (bottom row). When examining the sensitivity of the eigenvalues to $\alpha$, the parameters are chosen from $1$ to $2$;  when examining the sensitivity to $\beta$, we consider $\beta$ from $-0.7$ to $0.7$.  We examine results for the half circle and quadrant. We see that for the change in both $\alpha$ and $\beta$, the first eigenvalue  shows the most change  regardless of the size of the scatter.

Next, we demonstrate finding Steklov eigenvalues from far field data. We first compute scattering data (approximating $\left.u_{\infty}\right)$, then for several choices of the Steklov parameter $\lambda$ we approximate the corresponding far field pattern $h_{\infty}$ and then solve a discrete analogue of the far field equation using Tikhonov regularization with a fixed Tikhonov parameter $\gamma=10^{-10}$. 

For the forward problem, we  {truncate}  the mesh by a Perfectly Matched Layer with constant absorption parameters.  We choose meshes where the triangle diameter is approximately $1/32$ of the wavelength
in the element. In each case we use $60$ incoming waves with directions $\bfd_{j}, j=1, \ldots, 60,$ uniformly on the unit circle and hence compute a $60 \times 60$ matrix $A$ with $A_{\ell, m} \approx u_{\infty}\left(\bfd_{\ell}, \bfd_{m}\right)$. Similarly we obtain a matrix $B$ approximating $h_{\infty}$. Setting the data vector $b_{z}$ with $\ell$ th entry {to be} $b_{z, \ell}= \phi_{\infty}\left(\bm d_{\ell}, \bfz\right), 1 \leq \ell \leq 60$ for some $\bfz \in B$. Finally we compute an approximation to the Herglotz kernel $g_z$ using Tikhonov regularization, setting $M= A - B$ and
$$
g_{z}=\left(\left(M\right)^{*} M+\gamma  I\right)^{-1}\left(M\right)^{*} b_{z},
$$
where the superscript $*$ denotes the conjugate transpose of the given matrix {and $\gamma$ is the Tikhonov regularization parameter}. This is repeated for $20$ randomly placed $\bfz \in B$ and the norm of $g_z$ is averaged.  This method follows the standard approach in, for example, \cite{CCMM2016}.

The choice of $60$ directions and $20$ points $\bfz \in B$ is essentially arbitrary, but we need sufficiently many incoming waves {to be able to approximate}  the far field operator. Our  results in \Cref{numer2} are for $\alpha=\beta=0$, peaks in the norm of $g_{z}$ correspond well to eigenvalues. Higher eigenvalues cannot be detected. 

\begin{figure}
	\begin{center}
		\begin{tabular}{cc}
			\resizebox{0.47\textwidth}{!}{\includegraphics{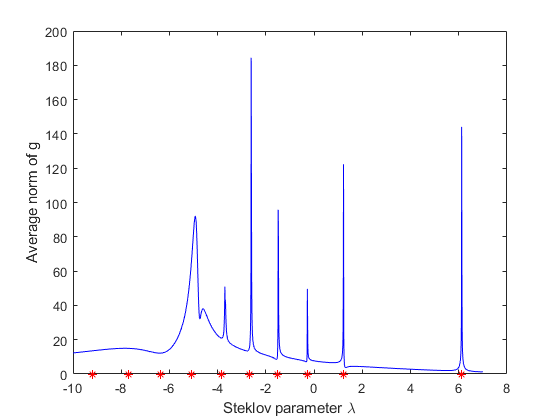}}&
			\resizebox{0.47\textwidth}{!}{\includegraphics{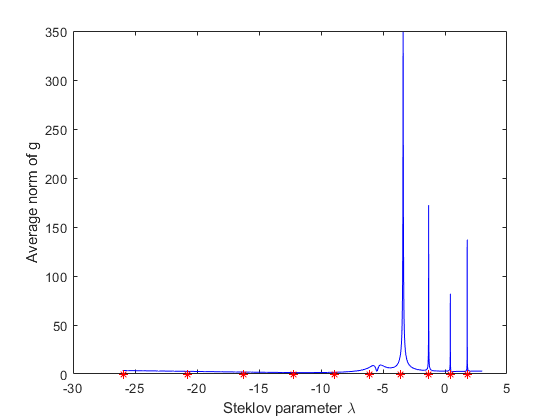}}
		\end{tabular}
	\end{center}
	\caption{The average of the $\ell^2$ norm of the discrete approximation $g_z$ against the Steklov parameter $\lambda$. The crosses in each figure show the position of Steklov eigenvalues computed by our finite element eigenvalue code. Left: semicircle,. Right: quadrant. Sharp peaks in the averaged norm of $g$ compute from far field data correspond to Steklov eigenvalues. 	}
	\label{numer2}	
\end{figure}

In  \Cref{numer3} we investigate detecting changes  by using $\alpha=0,0.5$ and $1$, and $\beta=0,\pm 0.2,\pm0.5$.  We can solve the far field equation as above, and determine shifts in Steklov eigenvalues from shifts in the peaks of the averaged norm of $g_{z}$. We have shown that Steklov eigenvalues can detect changes.

\begin{figure}
	\begin{center}
		\begin{tabular}{cc}
			\resizebox{0.47\textwidth}{!}{\includegraphics{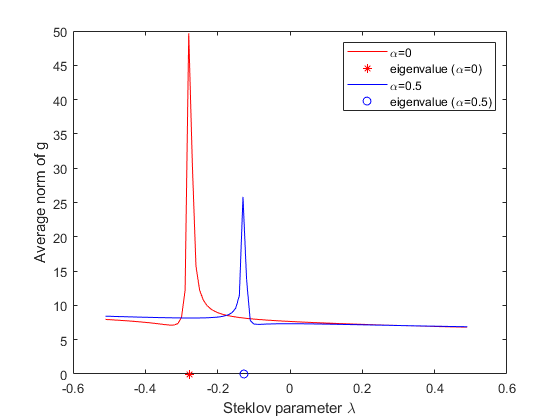}}&
			\resizebox{0.47\textwidth}{!}{\includegraphics{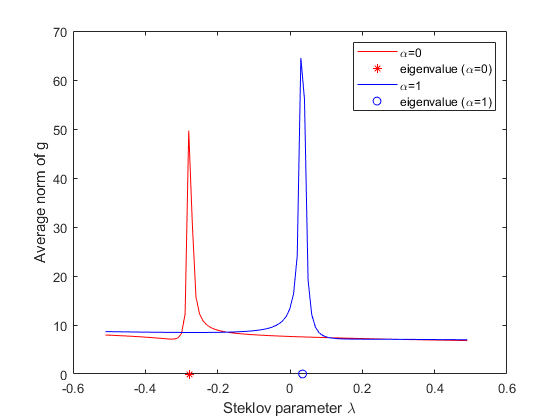}}\\
			\resizebox{0.47\textwidth}{!}{\includegraphics{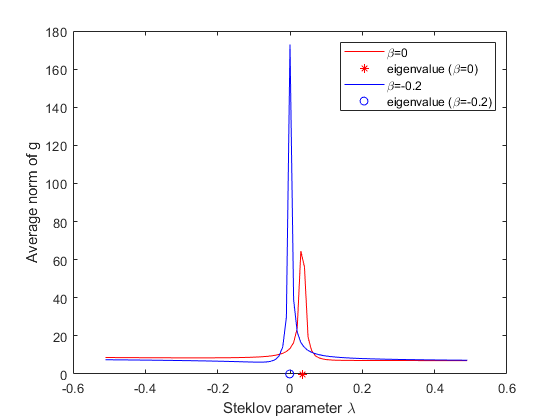}}&
			\resizebox{0.47\textwidth}{!}{\includegraphics{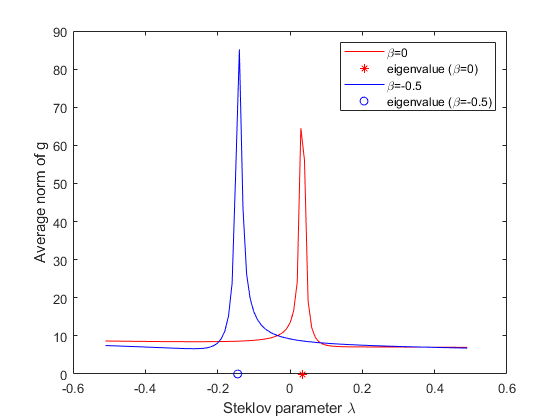}}\\
			\resizebox{0.47\textwidth}{!}{\includegraphics{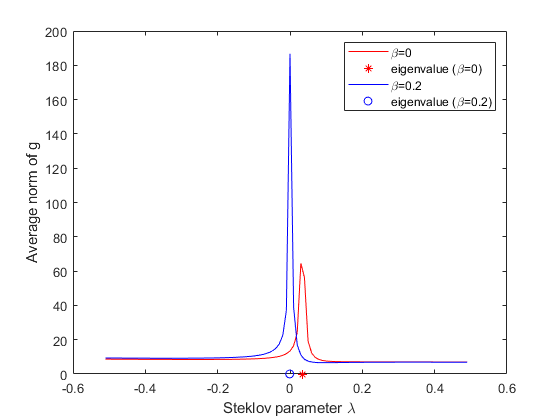}}&
			\resizebox{0.47\textwidth}{!}{\includegraphics{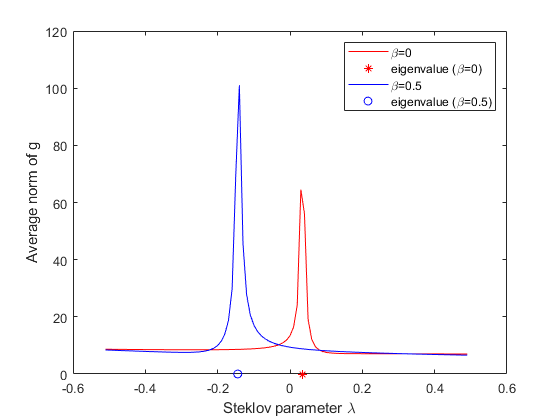}}
		\end{tabular}
	\end{center}
	\caption{Results for detecting changes in the eigenvalues due to changes in $\sigma$ (via changes in $\alpha$ and $\beta$). Top: Changing with respect to $\alpha$ by fixing $\beta=0$. Middle and bottom: Changing with respect to $\beta$ by fixing $\alpha=1$.}	
	\label{numer3}
	
\end{figure}

{For complex valued $\sigma$ the existence of eigenvalues has not been proved and we have not tested this case numerically.  Assuming such eigenvalues exist, we conjecture that by using a complex Steklov parameter $\lambda$ we could detect them by extending the search domain to include a region of the complex plane containing the eigenvalues as was done for modified Steklov eigenvalues in~\cite{CMMC}. }

\section{Conclusion}\label{sec:concl}
In this work, we studied the inverse scattering problem of obtaining target signatures for a screen.  The method is based on invoking an artificial domain such that the boundary of this domain contains  the screen.  Then we use the far field pattern of an auxiliary scattering problem having a mixed impedance and Neumann boundary condition to modify the far field operator for the screen. This results in target signatures that are the eigenvalues of a mixed Steklov eigenvalue problem. We analyze the problem when the function $\sigma$ describing the surface properties of the screen is real. 

It would be interesting to investigate this problem when $\sigma$ is complex, and analyze the case more elaborate models for the screen.  Another issue is the optimal choice of the artificial domain such that the mixed Steklov eigenvalues are sensitive to changes in $\sigma$, Finally, more complex models for the change in $\sigma$ need to be tested.

\section*{Acknowledgements}
P.\  Monk and Y.\ Zhang are partially supported by the US National Science Foundation (NSF) under grant DMS-1818867.   F.Cakoni and P. Monk are partially supported by AFOSR grant FA9550-20-1-0124. F. Cakoni is partially supported by US National Science Foundation (NSF) under grants DMS-1813492 and DMD- 2106255. F. Cakoni also acknowledges support from the NSF HDR TRIPODS award CCF-1934924.

\bibliographystyle{abbrv}
\bibliography{Refs}

\appendix
\section{Series solution of the mixed Steklov eigenvalue problem}\label{sec:exact}
We use the results in Section \ref{A1} below to test our Steklov eigenvalue solver.
\subsection{Sector of a circle} \label{A1} This example gives the exact mixed Steklov eigenvalues for the domain shown in the left panel of Fig.~\ref{Domains_Stoklev_eigenvalue_problem}.
We consider the sector of an circle given in polar coordinates by
\[
D_\alpha=\{(r,\theta)\;|\; 0<r<r_2,\; 0<\theta<\alpha\}		
\]
for some $2\pi>\alpha>0$.  {On} the segment of the boundary given by {$r=r_2$}, $0<\theta<\alpha$ we impose the Steklov boundary condition
\begin{equation}
	\partial_{\bfnu} u=\lambda^{(0)} u\label{S1a},
\end{equation}
and on the remaining boundaries we impose the homogeneous Neumann boundary condition. The superscript $(0)$ refers to the fact that we set $\sigma=0$ in (\ref{Steklov_eig}).

Separation of variables gives the following solution which satisfies the Neumann boundary condition for 
$\theta=0,\alpha$
\[
u(r,\theta)=\cos(n\pi{\theta}/\alpha)J_{n\pi/\alpha}(kr).	
\]
Satisfaction of the Steklov boundary condition gives the following equation for the Steklov eigenvalues
\[
\lambda^{(0)}_n=k\frac{(J_{n\pi/\alpha})'(kr_2)}
{J_{n\pi/\alpha}(kr_2)}.		
\]
For a given constant $\sigma$ the observed eigenvalues are
\begin{equation}
	\lambda_n=\sigma+\lambda^{(0)}_n.\label{trans}
\end{equation}
In Table \ref{tab0} we show the first nine eigenvalues $\lambda^{(0)}_n$ for the half-circle considered in \Cref{sec:num} where $r_1=1$, $\alpha=\pi$ and $k=2$.  Corresponding plots of the first six eigenfunctions are shown in Fig.~\ref{Fig0}. 
\begin{table}
	\begin{center}
		\begin{tabular}{cccccccccc}
			$\lambda^{(0)}_0$&$\lambda^{(0)}_1$&$\lambda^{(0)}_2$ &$\lambda^{(0)}_3$&$\lambda^{(0)}_4$&$\lambda^{(0)}_5$  &$\lambda^{(0)}_6$ &$\lambda^{(0)}_7$&$\lambda^{(0)}_8$\\\hline
			-5.1518  & -0.2236 &   1.2691 &   2.4727 &   3.5859  &  4.6584  &  5.7090   & 6.7464 &   7.7753\end{tabular}
		\caption{First 9 eigenvalues for a sector with $k=2$, $\alpha=\pi$ and $r_2=1$. }
	\end{center}
	\label{tab0}
\end{table}

\begin{figure}
	\begin{center}
		\begin{tabular}{c|c}
			\resizebox{0.4\textwidth}{!}{\includegraphics{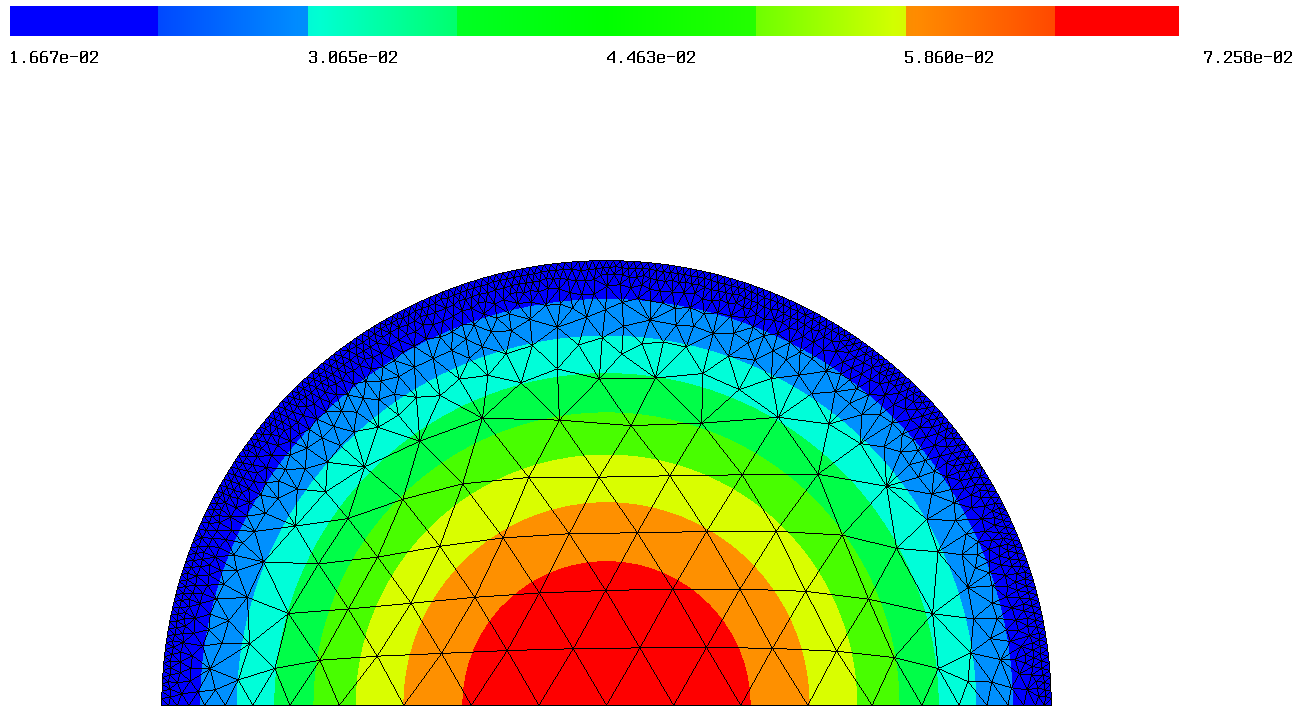}}&
			\resizebox{0.4\textwidth}{!}{\includegraphics{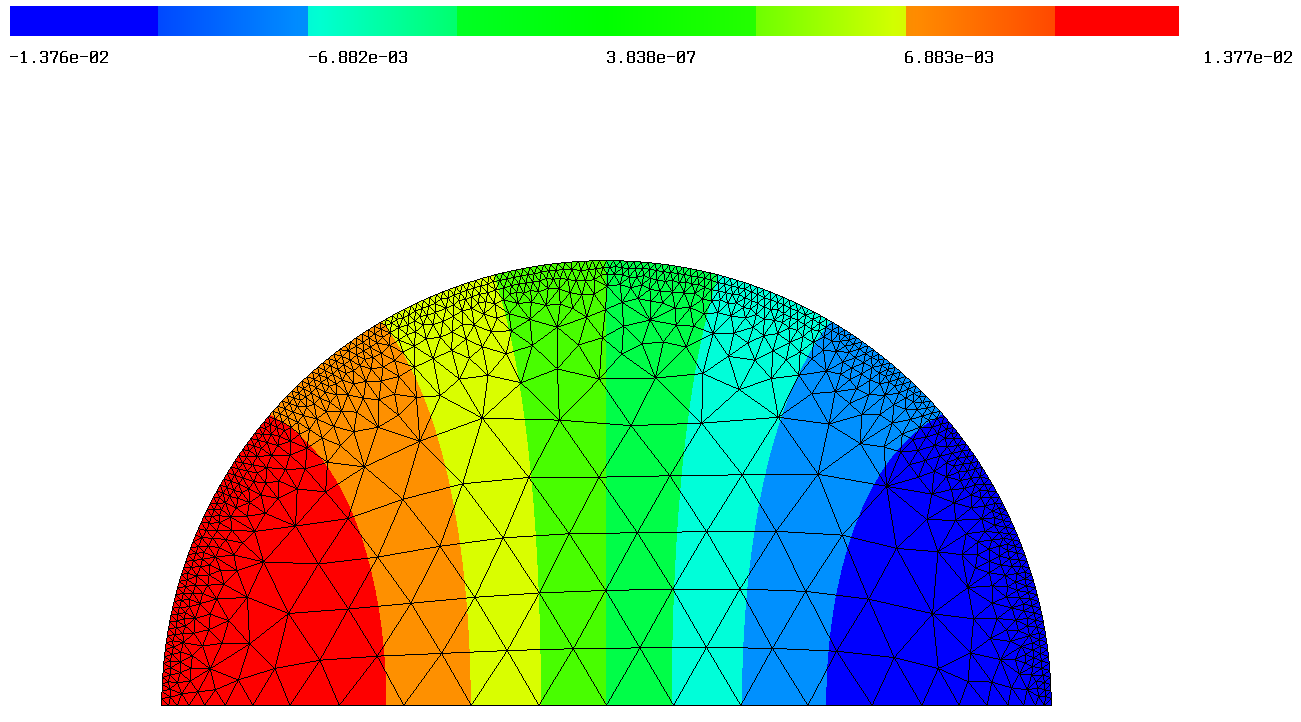}}\\\hline
			\resizebox{0.4\textwidth}{!}{\includegraphics{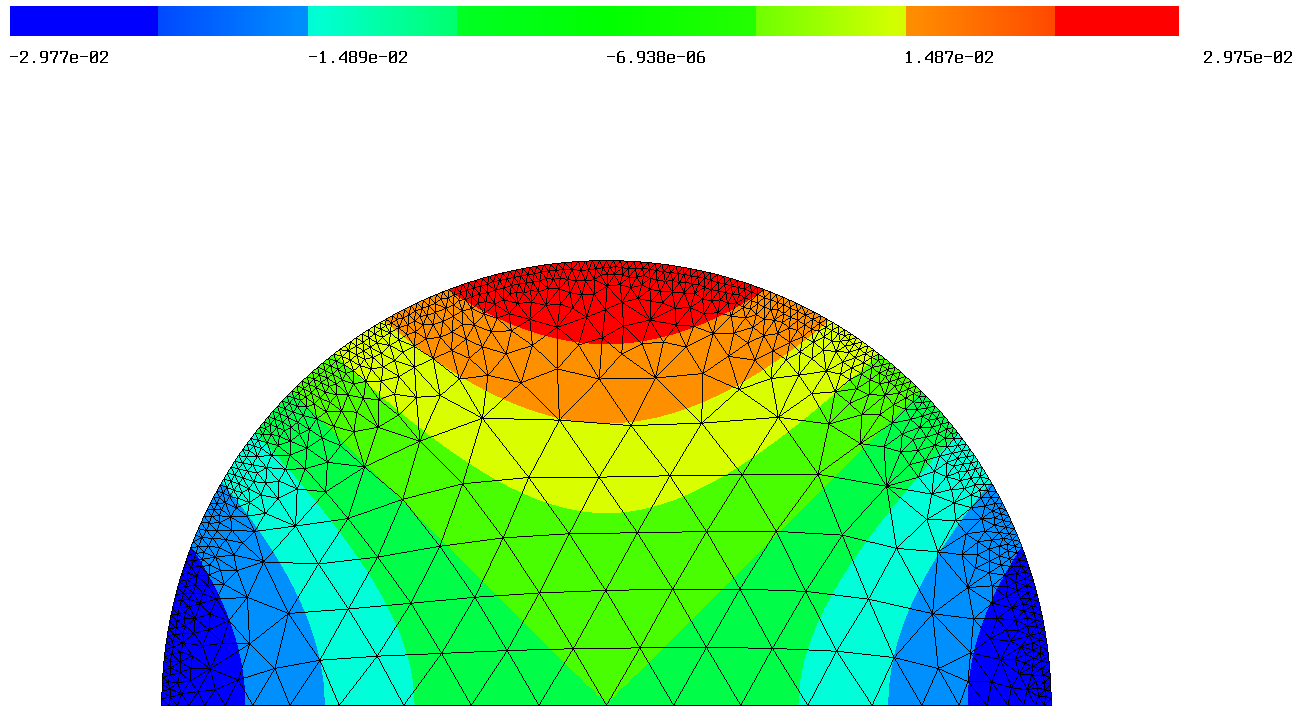}}&
			\resizebox{0.4\textwidth}{!}{\includegraphics{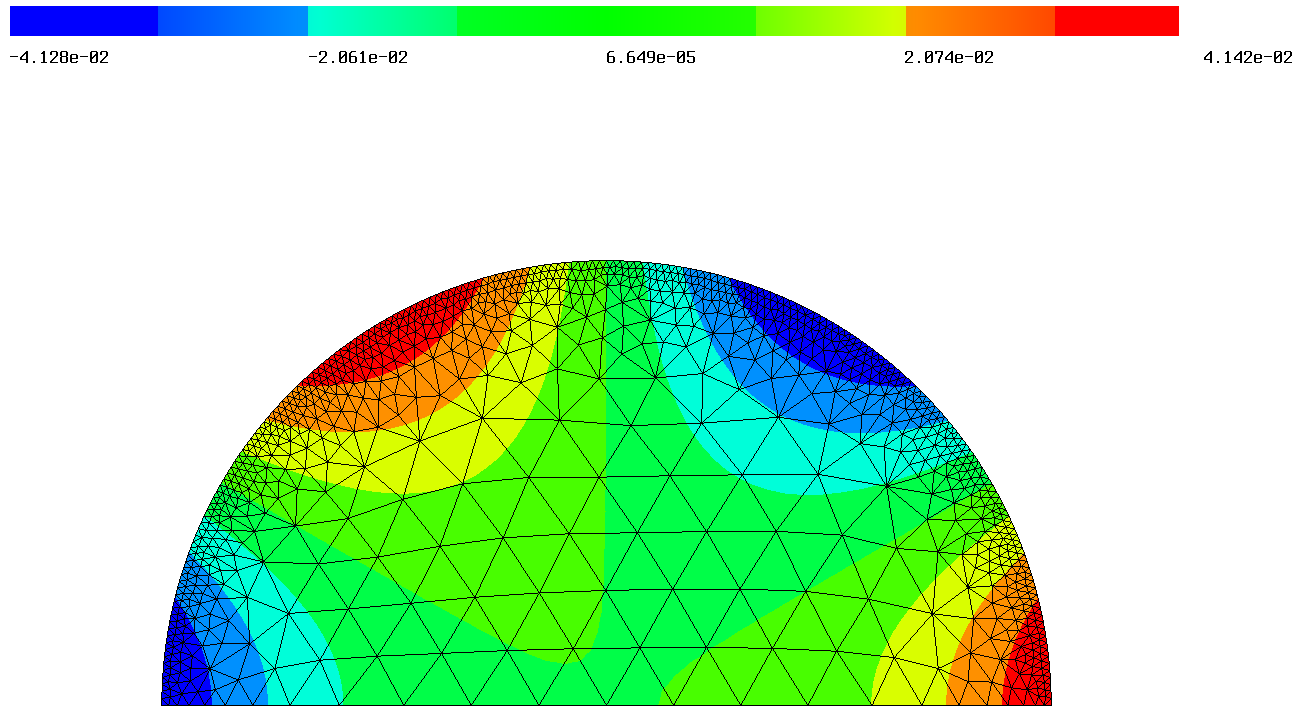}}\\\hline
			\resizebox{0.4\textwidth}{!}{\includegraphics{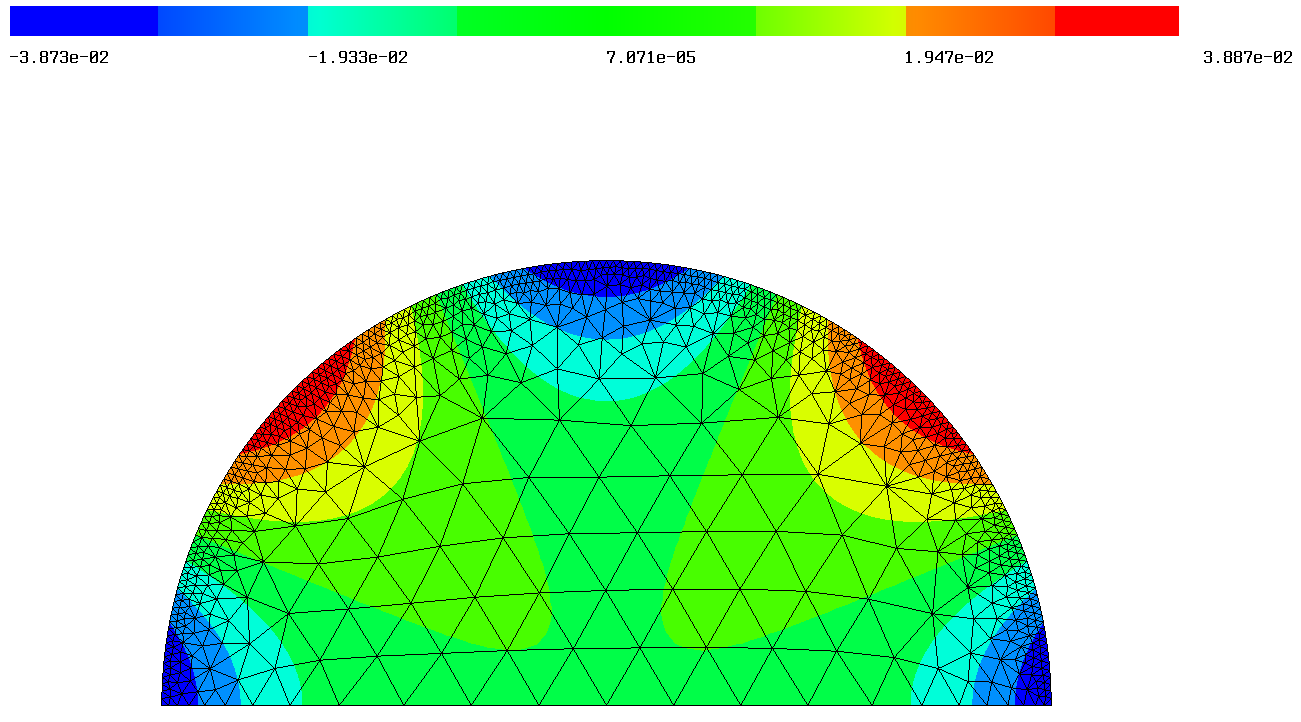}}&
			\resizebox{0.4\textwidth}{!}{\includegraphics{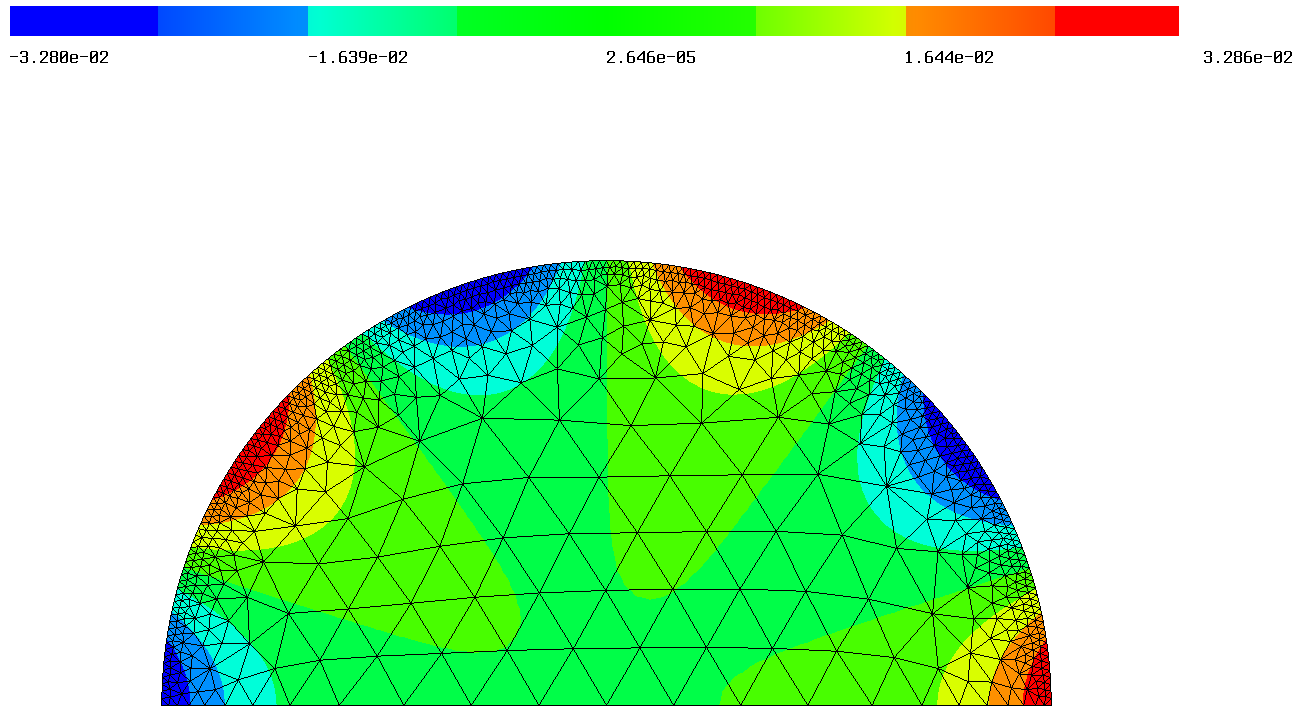}}\\
		\end{tabular}
	\end{center}
	\caption{First six eigenfunctions corresponding to the eigenvalues in Table~\ref{tab0}.}\label{Fig0}
\end{figure}

\subsection{{Angular} sector}
The domain considered here is not used in our study, but shows that a thin domain $D$ can support eigenfunctions along the screen.
We consider the sector of an annulus given in polar coordinates by
\[
D_\alpha=\{(r,\theta)\;|\; r_1<r<r_2,\; 0<\theta<\alpha\},
\]
for some $2\pi>\alpha>0$.  {On} the segment of the boundary given by $r=r_1$, $0<\theta<\alpha$ we impose the Steklov boundary condition (\ref{S1a})
and on the remaining boundaries we impose the Neumann boundary condition.  
Separation of variables gives the following solution which satisfies the Neumann boundary condition for 
$\theta=0,\alpha$
\[
u(r,\theta)=\cos(n\pi\theta/\alpha)(C_1H_{n\pi/\alpha}^{(1)}(kr)+ C_1H_{n\pi/\alpha}^{(2)}(kr)),
\]
Satisfaction of the Neumann boundary condition on $r=r_1$ gives
\[
k(C_1 (H_{n\pi/\alpha}^{(1)})'(kr_1)+ C_2(H_{n\pi/\alpha}^{(2)})'(kr_1))=0.
\]
So with this constraint
\[
u(r,\theta)=C_3\cos(n\pi/\alpha)((H_{n\pi/\alpha}^{(2)})'(kr_1)) H_{n\pi/\alpha}^{(1)}(kr)- (H_{n\pi/\alpha}^{(1)})'(kr_1)H_{n\pi/\alpha}^{(2)}(kr),
\]
Satisfaction of the Steklov boundary condition gives the following equation for the Steklov eigenvalues
\[
\lambda^{(0)}_n=k\frac{(H_{n\pi/\alpha}^{(2)})'(kr_1)) (H_{n\pi/\alpha}^{(1)})'(kr_2)- (H_{n\pi/\alpha}^{(1)})'(kr_1)(H_{n\pi/\alpha}^{(2)})'(kr_2)}
{(H_{n\pi/\alpha}^{(2)})'(kr_1)) H_{n\pi/\alpha}^{(1)}(kr_2)- (H_{n\pi/\alpha}^{(1)})'(kr_1)H_{n\pi/\alpha}^{(2)}(kr_2)}.
\]
For a given constant $\sigma$ the observed eigenvalues are then given by (\ref{trans}).

In Table~\ref{tab1} we show the first nine eigenvalues $\lambda^{(0)}_n$ for the sector where $r_2=1$, $\alpha=\pi/2$ and $k=2$ for two different values of $r_1$.  Corresponding plots of the first six eigenfunctions are shown in Fig.~\ref{Fig1}. 

\begin{table}
	\begin{tabular}{l|cccccccccc}
		$r_1$&$\lambda^{(0)}_0$&$\lambda^{(0)}_1$&$\lambda^{(0)}_2$ &$\lambda^{(0)}_3$&$\lambda^{(0)}_4$&$\lambda^{(0)}_5$  &$\lambda^{(0)}_6$ &$\lambda^{(0)}_7$&$\lambda^{(0)}_8$\\\hline
		0.8& -0.7565 &  0.1692 &  2.3567&   4.8812&   7.3089  & 9.5760 & 11.7270&  13.8097 & 15.8557 \\
		0.9 &- 0.3849  & 0.0413  & 1.2482  & 3.0534 &  5.2381  & 7.6118  &10.0414  &12.4500&  14.8025
	\end{tabular}
	\caption{First 10 eigenvalues for an annular sector with $k=2$, $\alpha=\pi/2$, and $r_2=1$. }
	\label{tab1}
\end{table}

\begin{figure}
	\begin{center}
		\begin{tabular}{cc}
			\resizebox{0.4\textwidth}{!}{\includegraphics{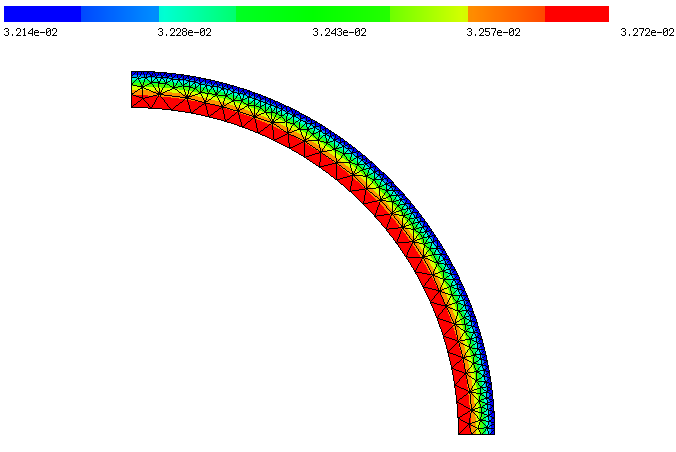}}&
			\resizebox{0.4\textwidth}{!}{\includegraphics{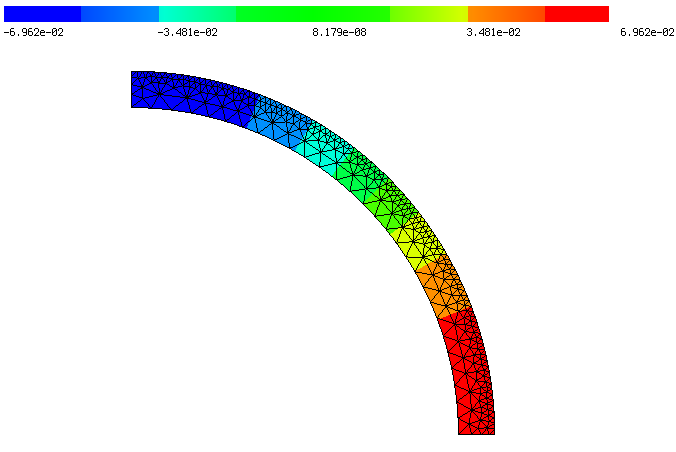}}\\
			\resizebox{0.4\textwidth}{!}{\includegraphics{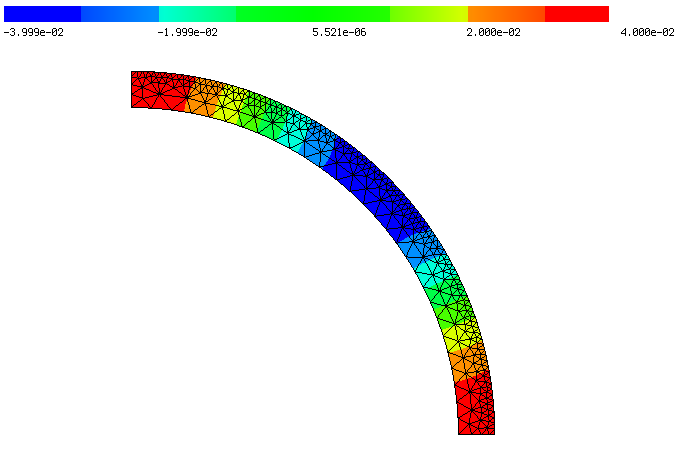}}&
			\resizebox{0.4\textwidth}{!}{\includegraphics{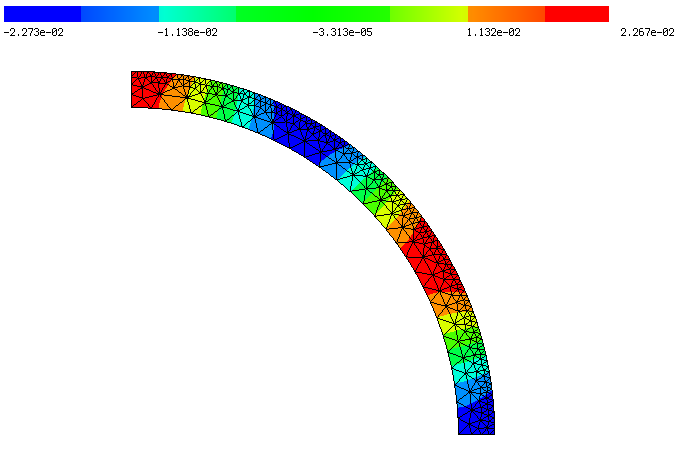}}\\
			\resizebox{0.4\textwidth}{!}{\includegraphics{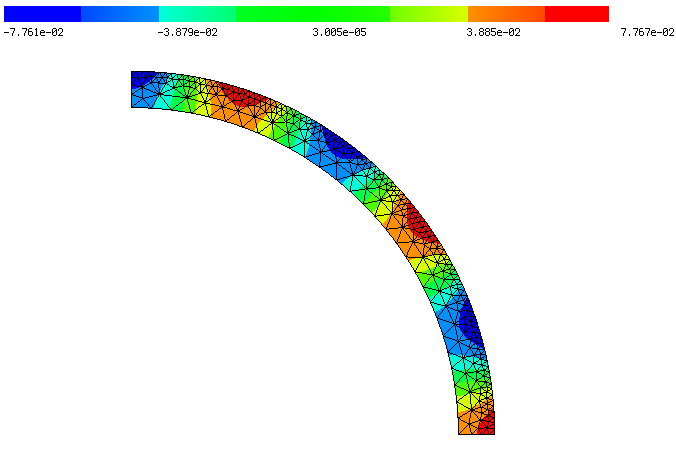}}&
			\resizebox{0.4\textwidth}{!}{\includegraphics{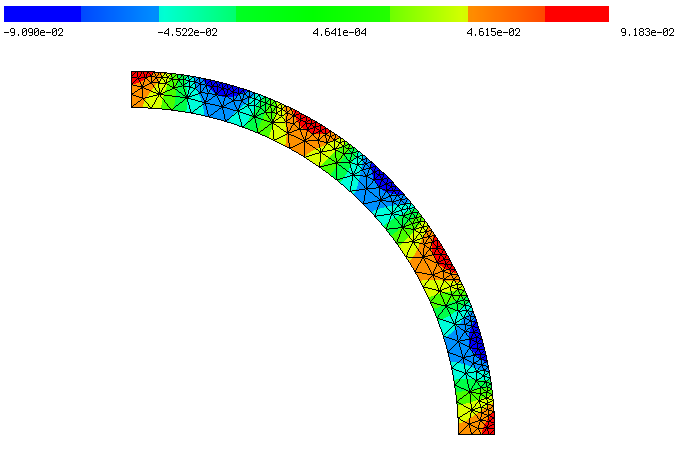}}\\
		\end{tabular}
	\end{center}
	\caption{Density plots of the first six eigenfunctions corresponding to the eigenvalues in Table~\ref{tab1} for $r_1=0.9$.}\label{Fig1}
\end{figure}

\end{document}